\documentclass[final,1p,times]{elsarticle}

\usepackage{graphicx}
\usepackage{amsmath}
\usepackage{amssymb}
\usepackage{color}
\usepackage{enumerate}
\usepackage{amsthm}

\usepackage{lineno}

\newcommand{\rmd}{\mathrm{d}}
\newcommand{\tr}{\mathrm{tr}}

\newtheorem{remark}{Remark}[section]
\newtheorem{theorem}{Theorem}[section]
\newtheorem{proposition}{Proposition}[section]
\newtheorem{corollary}{Corollary}[section]

\journal{Elsevier (To appear as J. Math. Anal. Appl. \textbf{495} (2021) 124673)}

\begin{document}

\begin{frontmatter}


\title{Invariant surfaces with coordinate finite-type Gauss map in simply isotropic space}



\author[AK]{Alev Kelleci\corref{cor}}
\ead{alevkelleci@hotmail.com}
\cortext[cor]{Corresponding author}

\author[LCBdS]{Luiz. C. B. da Silva}
\ead{luiz.da-silva@weizmann.ac.il}

\address[AK]{Department of Mathematics, F\i rat University, 23200 Elazi\u{g}, Turkey}

\address[LCBdS]{Department of Physics of Complex Systems, Weizmann Institute of Science, Rehovot 7610001, Israel}

\begin{abstract}
We consider the extrinsic geometry of surfaces in simply isotropic space, a three-dimensional space equipped with a rank 2 metric of index zero. Since the metric is degenerate, a surface normal cannot be unequivocally defined based on metric properties only. To understand the contrast between distinct choices of an isotropic Gauss map, here we study surfaces with a Gauss map whose coordinates are eigenfunctions of the surface Laplace-Beltrami operator.  We take into account two choices, the so-called minimal and parabolic normals, and show that when applied to simply isotropic invariant surfaces the condition that the coordinates of the corresponding Gauss map are eigenfunctions leads to planes, certain cylinders, or surfaces with constant isotropic mean curvature. Finally, we also investigate (non-necessarily invariant) surfaces with harmonic Gauss map and show this characterizes constant mean curvature surfaces.
\end{abstract}


\end{frontmatter}




\section{Introduction}

Let $M^n$ be a connected $n$-dimensional submanifold in the $m$-dimensional \linebreak
Euclidean space $\mathbb E^{m}$. We say that $M$ is of \emph{$k$-type} if its position vector $\mathbf{x}$ can be expressed as a sum of eigenvectors of the Laplace-Beltrami operator, $\Delta$, corresponding to $k$ distinct eigenvalues, i.e., 
$\mathbf{x}=\mathbf{x}_0+\mathbf{x}_1+\dots+\mathbf{x}_k,$ for a constant vector $\mathbf{x}_0$ and smooth non-constant functions $\mathbf{x}_k$, $(i=1,\ldots,k)$ such that $\Delta \mathbf{x}_i =\lambda_i \mathbf{x}_i$, $\lambda_i \in \mathbb{R}$,  \cite{C1984}. Several results concerning this subject can be found, e.g., in \cite{AGM1998,C1987,CMN1986,DPV1990,G1988}. (See \cite{C2014,C1996} for a survey in $\mathbb{E}^m$.) 

In \cite{T1966}, Takahashi proved that a submanifold $M^n$ in $\mathbb E^m$ is of $1$-type, i.e., $-\Delta \mathbf{x}=\lambda \mathbf{x}$,  if and only if it is either a minimal submanifold of $\mathbb E^m$ ($\lambda=0$) or a minimal submanifold of the hypersphere $\mathbb S^{m-1}\subset\mathbb E^m$ ($\lambda \neq 0$). As a generalization, in \cite{G1990}, Garay proved that if a hypersurface $M^n$ of $\mathbb E^{n+1}$ satisfies 
\begin{equation} \label{Adelta}
-\Delta \mathbf{x}= A\mathbf{x},
\end{equation}
where $A$ is a diagonal matrix $A=\mbox{diag}(\lambda_1,\dots,\lambda_{n+1})$, i.e., the coordinate functions of $M^n$ are eigenfunctions of $\Delta$ with possibly distinct eigenvalues, then it is a minimal hypersurface or an open piece of either round spheres or generalized right spherical cylinders. If an immersion satisfies Eq. \eqref{Adelta}, the submanifold is said to be of \emph{coordinate finite-type} \cite{HV1991}. Very recently, Senoussi and Bekkar studied helicoidal surfaces in $\mathbb{E}^3$ of coordinate finite-type \cite{SB2015}. Furthermore, coordinate finite-type submanifolds in pseudo-Euclidean spaces have been studied in \cite{AFL1992,BZ2008,HB2009}.

On the other hand, coordinate finite-type submanifolds in Cayley-Klein spaces equipped with a degenerate metric have taken attention of many geometers. For example, in Galilean and simply isotropic spaces, see \cite{D2012,D2013}  and \cite{AE2017,BKY2016,BKY2017,CKK2019,KYB2016,KYB2017,KYK2017,KYY2017}, respectively. In particular, we mention the classification of  revolution \cite{KYB2017} and helicoidal \cite{KYK2017} surfaces in isotropic 3-space.

The notion of finite-type submanifolds were generalized by studying the so-called submanifolds with finite-type Gauss map in \cite{CMN1986,CP1987}. In particular, a submanifold of (pseudo-)Euclidean space has 1-type Gauss map if and only if its Gauss map $\mathbf{G}$ satisfies $-\Delta \mathbf{G}= \lambda \mathbf{G}$ for $\lambda \in \mathbb R$. In Euclidean 3-space, a surface with 1-type Gauss map must necessarily be a plane, a circular cylinder, or a sphere \cite{CP1987,J1996}. As a generalization of this condition, Dillen \emph{et al.} \cite{DPVG1990}, Baikoussis and Blair \cite{BB1992}, and Baikoussis and Verstraelen
\cite{BV1993} respectively studied revolution, ruled, and helicoidal surfaces in Euclidean space which satisfy 
\begin{equation} \label{AGdelta}
-\Delta \mathbf{G}= A \mathbf{G},
\end{equation}
where $A=\mbox{diag}(\lambda_1,\lambda_2,\lambda_{3})$. The surface is then said to have \emph{coordinate finite-type Gauss map}. The theory of Gauss map of finite-type was also extended to Lorentzian \cite{AFL1998,C1995,CR1995}  and (pseudo-) Galilean spaces \cite{D2013,D2015}. 

Our goal is to investigate surfaces with coordinate finite-type Gauss map in simply isotropic space $\mathbb{I}^3$. However, unlike surfaces in $\mathbb{E}^3$, an isotropic surface normal cannot be unequivocally defined based on metric properties only. Indeed, the most natural choice would be to define the normal with respect to the ambient degenerate metric, which leads to the constant vector field $\mathcal{N}=(0,0,1)$ pointing in the isotropic direction. Instead, we shall consider two alternatives, either by mimicking the Euclidean approach in defining a normal $\mathbf{N}_m$ using a cross-like product or by imposing that the normal $\mathbf{G}$ takes values on a unit sphere of parabolic type, see Eqs. \eqref{eq::DefMinimalNormal} and \eqref{pgm}, respectively.  Here, we characterize simply isotropic invariant surfaces with coordinate finite-type parabolic $\mathbf{G}$ and minimal $\mathbf{N}_m$ Gauss maps.

The remaining of this work is divided as follows. After preliminaries results on isotropic geometry, Sect. \ref{SectPre}, and on isotropic invariant surfaces, Sect. \ref{sectInvarSurf}, we characterize helicoidal and parabolic revolution surfaces with coordinate finite-type Gauss maps in Sects. \ref{S:ClassificationHelicoidal} and \ref{S:ClassificationParabRev}, respectively. In Sect. \ref{S:ClassificationHarmonicGaussMaps}, we address the problem of characterizing (non-necessarily invariant) isotropic surfaces  with harmonic Gauss maps. Finally, in the last section, we present our concluding remarks along with suggestions for further lines of investigation.

\section{Preliminaries: Differential Geometry in Simply Isotropic Space}\label{SectPre}
First, we would like to give a brief summary of basic definitions, facts, and equations in the theory of surfaces in simply isotropic 3-space  (see for detail Sach's book \cite{Sachs1990}).

The simply isotropic 3-space $\mathbb I^3$ arises as a Cayley-Klein geometry whose
absolute figure in the 3-dimensional real projective space $\mathbb P^3(\mathbb R)$ is given by $\left\{\omega,d_1,d_2,F\right\}$. 
Here, homogeneous coordinates $[x_0: x_1: x_2: x_3]$ are introduced such that $\omega:x_0=0$ is a plane in $\mathbb P^3(\mathbb R)$, $d_1:x_0=0=x_1+\mathrm{i} x_2$ and $d_2:x_0=0=x_1-\mathrm{i}x_2$ are two complex-conjugate straight lines in $\omega$, and $F=[0:0:0:1]$ is a point in the intersection $d_1\cap d_2$.

The group of rigid motions of $\mathbb{I}^3$ comes from the projectivies of $\mathbb{P}^3(\mathbb{R})$ that leave the absolute figure invariant. Introducing affine coordinates, it is given by a six-parameter group $\mathcal{B}_6$ of affine transformations $(x,y,z) \mapsto (\bar{x},\bar{y},\bar{z})$ in $\mathbb R^3$ given by
\begin{eqnarray} \nonumber
\bar{x} &=& a + x\cos \phi - y \sin \phi,\\ \label{B6}
\bar{y} &=& b + x\sin \phi +y\cos \phi,\\\nonumber
\bar{z} &=& c + c_1 x+c_2 y+ z,
\end{eqnarray}
where $\phi,a,b,c,c_1,c_2 \in \mathbb R$. Regarding this group of isotropic motions, they appear as Euclidean motions onto the $xy$-plane. The projection of a point $P(x,y,z)$ on the $xy$-plane, $\widetilde{P}(x,y,0)$, is called the \emph{top view projection} of $P$. Let $\mathbf{X}=(x_1,x_2,x_3)$ be a vector in $\mathbb I^3$. If $x_1=x_2=0$, then $\mathbf{X}$ is said to be \emph{isotropic}, otherwise it is \emph{non-isotropic}. A line with an isotropic director is an \emph{isotropic line} and a plane containing an isotropic line is an \emph{isotropic plane}.

Given two vectors $\mathbf{X}=(x_1,x_2,x_3)$ and $\mathbf{Y}=(y_1,y_2,y_3)$, the \emph{isotropic inner product} is calculated by 
 \begin{equation}
\left\langle \mathbf{X},\mathbf{Y}\right\rangle=x_1y_1+x_2y_2.     
 \end{equation}
The \emph{isotropic distance} between two points $P_i=(x_i,y_i,z_i)$  with $i\in\left\{1,2\right\}$ is defined by 
$
{d}(P_1,P_2)=\sqrt{(x_2-x_1)^2+(y_2-y_1)^2}.
$ If two points have the same top views, then they are said to be \emph{parallel}. The isotropic inner product between parallel points vanishes identically. In this case, we introduce the isotropic  \emph{co-distance} $\mbox{cd}((a,b,x_3),(a,b,y_3))=\vert y_3-x_3\vert$.

When dealing with surfaces $M^2$ in isotropic geometry we must distinguish between two cases depending on whether the induced metric is degenerate or not. We say that $M^2$ is an \emph{admissible surface} when the metric in $M^2$ induced by the isotropic scalar product has rank 2. If $M^2$ is parameterized by a $C^2$ map $\textbf{x}(u^1,u^2)=\left(x^1(u^1,u^2),x^2(u^1,u^2),x^3(u^1,u^2)\right)$, then it is admissible if and only if $X_{12}=x_1^1 x_2^2-x_2^1 x_1^2 \neq 0$, where $x_k^i= {\partial x^i}/{\partial u^k}$ and 
\begin{align} \label{xij}
\begin{split}
X_{ij}=\mbox{det}\left( \begin{array}{cc}
x_1^i&x_1^j\\
x_2^i&x_2^j
\end{array}\right).
\end{split}
\end{align}
As a consequence, every admissible $C^2$ surface $M^2$ can be locally parameterized as $\textbf{x}(u^1,u^2)=\left(u^1,u^2,f(u^1,u^2)\right)$: we say that $M$ is in its \emph{normal form}. 

The  \emph{isotropic first fundamental form} $\mathrm{I}$ and the coefficients of the \emph{isotropic metric} tensor $g_{ij}$  are given by
\begin{equation}
    \mathrm{I}=g_{ij}\rmd u^i\rmd u^j\mbox{ and }g_{ij}=\langle\mathbf{x}_i,\mathbf{x}_j\rangle,
\end{equation}
where we are adopting the convention of summing on repeated indexes. In the normal form, the first fundamental form becomes $\mathrm{I}=(\rmd u^1)^2+(\rmd u^2)^2$.

\subsection{Extrinsic geometry in simply isotropic space}

Unlike surfaces in Euclidean space, where we may define curvatures through the behavior of the Gauss map defined as the unit normal of the surface, in simply isotropic space this is not possible since the normal with respect to the isotropic metric is the constant vector field $\mathcal{N}=(0,0,1)$. However, the concept of Christoffel symbols $\Gamma_{ij}^k$ and the second fundamental form $\mathrm{II}=h_{ij}\rmd u^i\rmd u^j$ are still meaningful. Indeed, for an admissible surface it is valid $\det(\mathbf{x}_1,\mathbf{x}_2,\mathcal{N})\not=0$ and then, we write
\begin{equation}
    \mathbf{x}_{ij} = \Gamma_{ij}^k\mathbf{x}_k+h_{ij}\mathcal{N}.
\end{equation}
To write the coefficients $h_{ij}$ in terms of an inner product, we may take two paths. On the one hand, we may use the Euclidean inner $\cdot$ and vector $\times$ products and write
\begin{equation}\label{eq::DefMinimalNormal}
    h_{ij}=\mathbf{x}_{ij}\cdot\mathbf{N}_m,\mbox{ where }\mathbf{N}_m=\frac{\mathbf{x}_1\times\mathbf{x}_2}{X_{12}}=(\frac{X_{23}}{X_{12}},\frac{X_{31}}{X_{12}},1).
\end{equation}
We shall call $\mathbf{N}_m$ the \emph{minimal normal} since the trace of the Weingarten-like operator $-\rmd\mathbf{N}_m$ vanishes identically. Indeed, introducing $\mathbf{a}_i=\mathbf{x}_i\times\mathcal{N}$, $\mathbf{N}_m$ satisfies the Weingarten-like equation (as given in \cite{Sachs1990}, p. 160),
\begin{equation}
    \frac{\partial\mathbf{N}_m}{\partial u^i}=\frac{h_{i2}}{\sqrt{g}}\mathbf{a}_1-\frac{h_{1i}}{\sqrt{g}}\mathbf{a}_2,
\end{equation}
where $g=\det(g_{ij})$. It is easy to see that $\tr(-\rmd\mathbf{N}_m)=0$. (Its determinant, however, is non-trivial and gives the Gaussian curvature to be defined below.)

On the other hand, we may impose that the isotropic Gauss map should take values on a unit sphere of parabolic type. More precisely, we first take
\begin{equation}
    \Sigma^2=\{(x,y,z)\in\mathbb{I}^3:z=\frac{1}{2}-\frac{x^2+y^2}{2}\}
\end{equation}
as the reference sphere. (In simply isotropic space, we may have spheres of parabolic and cylindrical types \cite{daSilvaTamkang,Sachs1990}, but only spheres of parabolic type are admissible.) Then, we define the \emph{isotropic Gauss map} as \cite{daSilvaJG2019}
\begin{equation} \label{pgm}
    \mathbf{G}(u^1,u^2)=\left(\frac{X_{23}}{X_{12}},\frac{X_{31}}{X_{12}},\frac{1}{2}-\frac{1}{2}\Big[\Big(\frac{X_{23}}{X_{12}}\Big)^2+\Big(\frac{X_{31}}{X_{12}}\Big)^2\Big]\right),
\end{equation}
from which we also define an \emph{isotropic shape operator} $S=-\rmd \mathbf{G}$ \cite{daSilvaJG2019}. We shall also refer to $\mathbf{G}$ as the \emph{parabolic normal}. Finally, the coefficients of the first and second  fundamental forms can be written as
\begin{equation} \label{ffc}
g_{ij}=\left\langle \textbf{x}_i,\textbf{x}_j \right\rangle\mbox{ and } h_{ij}=\mathrm{II}(\textbf{x}_i,\textbf{x}_j)=\mathrm{I}(S(\mathbf{x}_i),\mathbf{x}_j).
\end{equation}
The \emph{isotropic Gaussian} and \emph{mean curvatures}, $K$ and $H$, are respectively defined as the determinant and trace of the shape operator $-(\rmd \mathbf{G})^i_j=g^{ik}h_{kj}$:
\begin{equation}
    K=\frac{h_{11}h_{22}-h_{12}^2}{g_{11}g_{22}-g_{12}^2}\mbox{ and } H=\frac{1}{2}\frac{g_{11}h_{22}-2g_{12}h_{12}+g_{22}h_{11}}{g_{11}g_{22}-g_{12}^2}.
\end{equation}

In order to understand the contrast between different choices of an isotropic Gauss map and to build some intuition, here we will study surfaces with coordinate finite-type Gauss map  using both $\mathbf{G}$ and $\mathbf{N}_m$, i.e., surfaces whose coordinates of the corresponding Gauss map are eigenfunctions of the Laplace operator. In terms of a local coordinate system, the Laplacian $\Delta$ is defined as usual by
\begin{equation} \label{lap}
\Delta=\frac{\partial_i\left(\sqrt{g}\,g^{ij}\partial_j\right)}{\sqrt{g}}=\frac{1}{\sqrt{g}}\left[{\partial_{1}}\Big(\frac{g_{22} \partial_{1}-g_{12} \partial_{2}}{\sqrt{g}}\Big)+{\partial_{2}}\Big(\frac{g_{11} \partial_{2}-g_{12} \partial_{1}}{\sqrt{g}}\Big)\right],
\end{equation}
where $\partial_{i}=\partial/\partial u^i$ and $g^{ij}$ is the inverse of the metric, i.e., $g^{ik}g_{kj}=\delta_j^i$.

\begin{remark}
The construction of the isotropic Gauss map employed above may be properly understood in the framework of the affine differential geometry. Indeed, many properties usually associated with the behavior of the unit normal of surfaces in Euclidean space can be extended to other contexts with the help of the notion of \emph{relative normal} \cite{Mueller1921,Simon1991}. Such construction requires the introduction of a vector field $\xi$ along a surface $M^2$ which is both (i) \emph{transversal} to $M^2$, i.e., $\xi$ is not tangent, and (ii) \emph{equiaffine}, i.e.,  $\rmd\xi$ is tangent.  The parabolic Gauss map $\mathbf{G}$ is a relative normal, but the same is not true for the minimal normal $\mathbf{N}_m$ since it is not equiaffine. In this latter case, we may see the introduction of the vector fields $\mathbf{a}_i=\mathbf{x}_i\times\mathcal{N}$ as an attempt to remedy this since $\mathbf{N}_m$ is transversal and equiaffine with respect to the distribution of planes $\mbox{span}\{\mathbf{a}_1,\mathbf{a}_2\}$. 
\end{remark}

\section{Simply Isotropic Invariant Surfaces}
\label{sectInvarSurf}

In this work, we will be mainly interested in invariant surfaces. Here, geometric quantities, such as the Gaussian and mean curvatures, only depend on their values assumed along the generating curve. In addition, the study of surfaces with coordinate finite-type Gauss map reduces to the analysis of ordinary differential equations. (The interested reader is referred to \cite{daSilvaMJOU2019} for more details on isotropic invariant surfaces.)

A \emph{1-parameter subgroup} $\mathcal{H}$ of the group $\mathcal{B}_6$ of isotropic isometries is given by a surjective continuous group homomorphism $\psi:(\mathbb{R},+)\to(\mathcal{B}_6,\circ)$, i.e., $\psi(0)=\mathrm{Id}$ is the identity rigid motion and $\psi(s+t)=\psi(s)\circ\psi(t)$. (It is common to denote $\psi_t=\psi(t)$ and, despite that $\psi$ is not unique, we may identify $\mathcal{H}$ with $\psi$ since $\psi(\mathbb{R})=\mathcal{H}$.) A surface $M^2$ is said to be \emph{invariant} if there exists a 1-parameter subgroup $\mathcal{H}$ such that $M=\psi_t(M)$ for all $t\in \mathbb{R}$. By intersecting an invariant surface with a plane, usually the $xz$- or the $xy$-plane, we obtain a curve $\alpha$, the \textit{generating curve} of $M$, and we can then  see $M^2$ as the result of continuously moving $\alpha$ under the action of $\psi_t$. In addition, we may parameterize $M^2$ as $\mathbf{x}(u,t)=\psi_t(\alpha(u))$. 

For simply isotropic rigid motions what happens in the top view plane is independent from what happens in the isotropic $z$-direction. Then, we may classify the 1-parameter subgroups based on their action on the top view plane and on the isotropic direction separately \cite{daSilvaMJOU2019,Sachs1990}. The 1-parameter subgroups of simply isotropic isometries can be distributed along 7 types, which are divided into two main categories: 

(a) \emph{helicoidal motions}, which in the isotropic direction act either as a pure translation or as the identity map:
\begin{equation}\label{eqHelicoidal1-parSubgroup}
t\in\mathbb{R}\mapsto \psi_t(\mathbf{x})=
\left(\arraycolsep=4pt\begin{array}{ccc}
\cos(t\phi) & -\sin(t\phi) & 0 \\
\sin(t\phi) & \cos(t\phi) & 0  \\
0 & 0 & 1 \\
\end{array}\right)
\left(\begin{array}{c}
x_1\\
x_2 \\
x_3\\
\end{array}\right)
+
\left(\begin{array}{c}
0\\
0 \\
c\,t\\
\end{array}\right);
\end{equation}

(b) \emph{limit motions} (\textit{Grenzbewegungen} \cite{Sachs1990}), which in the top view plane act either as a pure translation or as the identity map:
\begin{equation}\label{eqParabolic1-parSubgroup}
t\in\mathbb{R}\mapsto \psi_t(\mathbf{x})=
\left(\arraycolsep=4pt\begin{array}{ccc}
1 & 0 & 0 \\[4pt]
0 & 1 & 0 \\[4pt]
c_1t & c_2t & 1 \\[4pt]
\end{array}\right)
\left(\begin{array}{c}
x_1\\
x_2 \\
x_3\\
\end{array}\right)
+
\left(\begin{array}{c}
a\,t\\
b\,t \\
c\,t+(ac_1+bc_2)\frac{t^2}{2}\\
\end{array}\right).
\end{equation}
The constants $\phi,a,b,c,c_1,$ and $c_2$ are the same as those appearing in Eq. (\ref{B6}).

Invariant surfaces obtained from helicoidal motions will be called  \emph{helicoidal surfaces} while those obtained from limit motions will be called \emph{parabolic revolution surfaces}. Notice that helicoidal surfaces are foliated by helices while parabolic revolution surfaces are foliated by isotropic circles, i.e., parabolas whose symmetry axis is an isotropic line. In addition, we shall restrict ourselves to  \emph{invariant surfaces of i-type} \cite{daSilvaMJOU2019}, i.e., the generating curve $\alpha$ comes from an intersection of the surface with the isotropic $xz$-plane.

\subsection{Helicoidal surfaces} 

Let the generating curve  $\alpha$ be parameterized by arc-length, $\alpha(u)=(u,0,z(u))$. A helicoidal surface $ M^2_c$ is then parameterized as
\begin{equation}\label{hs1}
 M^2_c:  \mathbf{R}(u,t) = (u\cos t, u\sin t, z(u)+ct),\,u>0.
\end{equation}
where for simplicity we set $\phi=1$. The first and second fundamental forms of a helicoidal surface are given by \cite{daSilvaMJOU2019}
\begin{equation}
\mathrm{I}=\rmd u^2+u^2\rmd t^2
\mbox{ and }
\mathrm{II}=z''\rmd u^2-\frac{2c\,\rmd u\rmd t}{u}+ uz'\rmd t^2,
\end{equation}
from which  follows that the Gaussian and mean curvatures are
\begin{equation}\label{eq::KandHhelicoidalSurfaces}
K=\frac{z'z''}{u}-\frac{c^2}{u^4}
\mbox{ and }
H = \frac{z'+uz''}{2u}.
\end{equation}
When $c=0$, we say that $M_0^2$ is a \emph{revolution surface}.

In addition, the minimal normal is
\begin{equation}
\mathbf{N}_m = (\frac{c}{u}\sin t-z'\cos t,-\frac{c}{u}\cos t-z'\sin t,1),
\end{equation}
while the parabolic normal is
\begin{equation} \label{pgmh}
\mathbf{G}=(\frac{c}{u}\sin t-z'\cos t,-\frac{c}{u}\cos t-z'\sin t,\frac{1}{2}\Big(1-\frac{c^2}{u^2}-{z'}^2\Big)).
\end{equation}

Finally, the Laplace-Beltrami operator of a helicoidal surface is given by
\begin{equation} \label{laph}
\Delta=\frac{1}{u} \frac{\partial}{\partial u}+ \frac{\partial^2}{\partial u^2}+ \frac{1}{u^2} \frac{\partial^2}{\partial t^2}.
\end{equation}
In particular, the Laplacian is the same for all helicoidal surfaces.

\subsection{Parabolic revolution surfaces} \label{Prs}

Let the generating curve  $\alpha$ be parameterized by arc-length, $\alpha(u)=(u,0,z(u))$. A parabolic revolution surface $M^2_{(a,b,c,c_1,c_2)}$ is parameterized as
\begin{equation}
 M^2_{(a,b,c,c_1,c_2)}:\mathbf{P}(u,t)=(at+u,bt,ct+\frac{ac_1+bc_2}{2}t^2+c_1ut+z(u)),\,u,b>0.
\end{equation}
The corresponding first and second fundamental forms are given by \cite{daSilvaMJOU2019}
\begin{equation}
\mathrm{I}=\mathrm{d}u^2+2a\mathrm{d}u\rmd t+(a^2+b^2)\rmd t^2
\mbox{ and }
\mathrm{II}=z''\rmd u^2+2c_1\rmd u\rmd t+(ac_1+bc_2)\rmd t^2,
\end{equation}
from which it follows that the Gaussian and mean curvatures are
\begin{equation}
K=\frac{(ac_1+bc_2)z''}{b^2}-\frac{c_1^2}{b^2}
\mbox{ and }
H = \frac{bc_2-ac_1}{2b^2}+\frac{(a^2+b^2)z''}{2b^2}.
\end{equation}
When $c=ac_1+bc_2=0$, but $(a,b),(c_1,c_2)\not=(0,0)$, we say that $M^2_{(a,b,0,c_1,c_2)}$ is a \emph{warped translation surface}. Moreover, when $c=c_1=c_2=0$, but $(a,b)\not=(0,0)$, we say that $M_{(a,b,0,0,0)}^2$ is a \emph{translation surface}.

In addition, the minimal normal of a parabolic revolution surface is
\begin{equation} \label{mnParRev}
\mathbf{N}_m = (-c_1t-z',\frac{az'-c-bc_2t-c_1u}{b},1),
\end{equation}
while the parabolic normal is
\begin{equation} \label{pgmParRev}
\mathbf{G}=(-\frac{bc_1t+bz'}{b},\frac{az'-c-bc_2t-c_1u}{b},G^3),
\end{equation}
where 
\begin{equation}
G^3 = \frac{1}{2} -\frac{(c+c_1u)^2}{2b^2}+\frac{a(c+c_1u)}{b^2}z'-\frac{a^2+b^2}{2b^2}z'^2+ \frac{t}{b} \Big[(a c_2-bc_1)z'-c_2(c+c_1u)\Big]-\frac{t^2}{2}
   \left(c_1^2+c_2^2\right). \label{eq::3rdCoordNormalParabRevSurf}
\end{equation}

Finally, the Laplace-Beltrami operator of a parabolic revolution surface is given by
\begin{equation} \label{eq::LapParabRevSurf}
\Delta=\frac{a^2+b^2}{b^2}\frac{\partial^2}{\partial u^2}-\frac{2a}{b^2}\frac{\partial^2}{\partial u\partial t}+\frac{1}{b^2}\frac{\partial^2}{\partial t^2}.
\end{equation}

\section{Helicoidal Surfaces with Coordinate Finite-type Gauss Map}
\label{S:ClassificationHelicoidal}

Since the top view projection of both parabolic and minimal Gauss maps coincide, we may start investigating the eigenvalue problem for the first two coordinates of the minimal Gauss map $\mathbf{N}_m$. Since $N_m^3=1$,  the eigenvalue problem associated with the third coordinate of $\mathbf{N}_m$ is trivial. On the other hand, this is not the case for the last coordinate of the parabolic normal $\mathbf{G}$. Thus, after characterizing the surfaces whose minimal normal is of coordinate finite-type, we will also know the solutions for the first two coordinates of the parabolic Gauss map $\mathbf{G}$. After that, the strategy to complete the study of $\mathbf{G}$ will consist in checking the compatibility of the eigensolutions of the first coordinates $\{G^1,G^2\}$ with the eigenvalue problem for the last coordinate $G^3$.

\subsection{Helicoidal surfaces with coordinate finite-type minimal normal}

The Laplacian of the minimal Gauss map $\mathbf{N}_m$ of a helicoidal surface is 
\begin{equation}
\Delta \mathbf{N}_m = \left(\frac{z'-u^2 z'''-u z''}{u^2}\cos t,\,\frac{z'-u^2 z'''-u z''}{u^2}\sin t,\,0\right).
\end{equation}

Now, we would like to classify helicoidal surfaces given by Eq. \eqref{hs1} in $\mathbb I^3$ satisfying the coordinate finite-type equation \eqref{AGdelta}:
\begin{equation}
-\Delta (N_m^1,N_m^2,N_m^3)=(\lambda_1N_m^1,\lambda_2N_m^2,0).
\end{equation} 
The corresponding eigenvalue problems become
\begin{equation*}
\cos t \Big(\frac{u^2 z'''+uz''-z^{\prime}}{u^2}\Big)=\lambda_1(\frac{c}{u}\sin t-z'\cos t)
\end{equation*}
and
\begin{equation*}
\sin t \Big(\frac{u^2 z'''+uz''-z^{\prime}}{u^2}\Big)=\lambda_2(-\frac{c}{u}\cos t-z'\sin t).
\end{equation*}
Now, using that $\{\cos t,\sin t\}$ is a set of linearly independent functions, we have the following equations for $\lambda_1$ and $\lambda_2$:
\begin{equation}\label{Eq::EigenvProbLamb1HSurf}
\left\{
\begin{array}{c}
\lambda_i c = 0\\[4pt]
u^2 z'''+uz''-(1-\lambda_i u^2)z'=0\\
\end{array}
\right.,\,i=1,2.
\end{equation}
Analyzing all possibilities, we have the following classification of helicoidal surfaces whose minimal normal $\mathbf{N}_m$ is of coordinate finite-type.

\begin{figure}[t]
    \centering
    \includegraphics[width=0.75\linewidth]{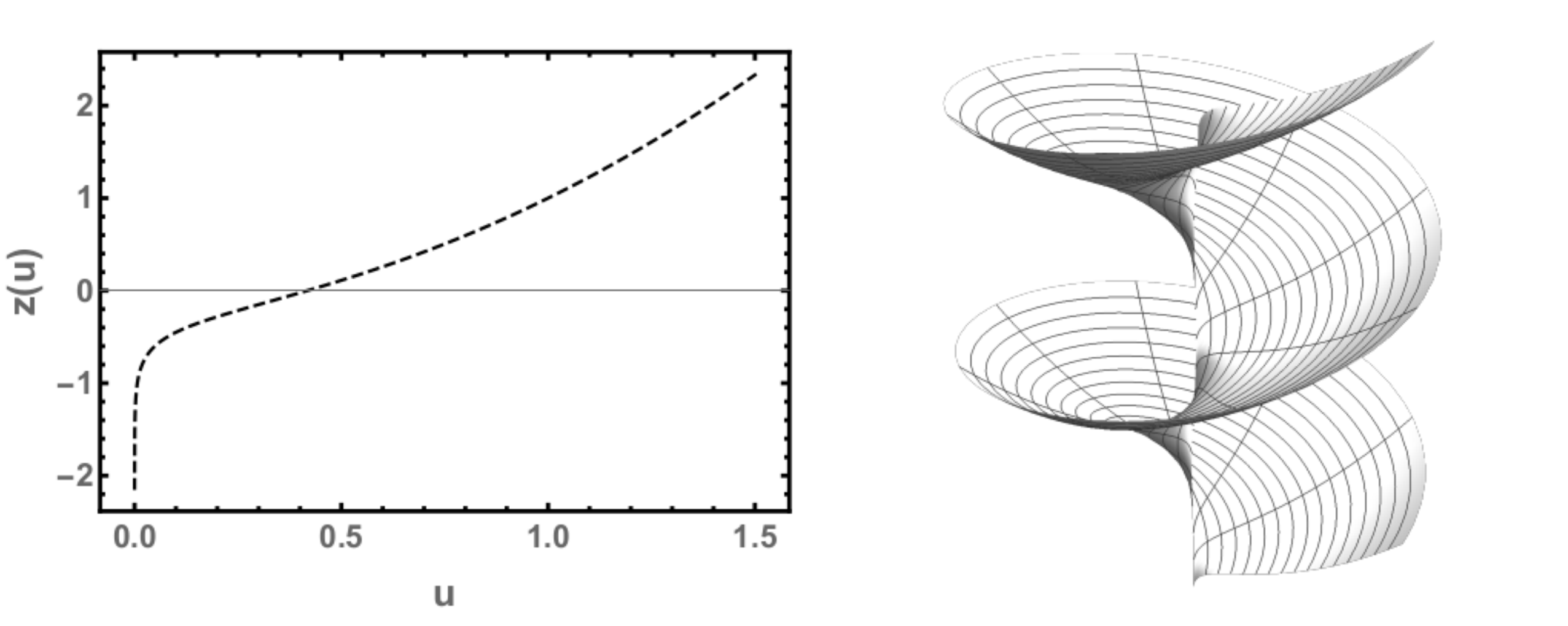}
    \caption{Helicoidal surfaces with harmonic minimal Gauss map $\mathbf{N}_m$ according to Theorem \ref{Thr::FiniteTypeTopViewGaussMapH}. These surfaces have constant mean curvature. (Left) Plot of $z(u)=z_0+z_1u^2+z_2\ln u$; (Right) Plot of the corresponding helicoidal surface. (In the figure, $u\in(0,\frac{3}{2})$, $t\in(0, 4\pi)$, $c=1$, $z_0=0$, $z_1=1$, and $z_2=\frac{1}{4}$.)}
    \label{fig:CMChelSurf}
%
    \includegraphics[width=0.8\linewidth]{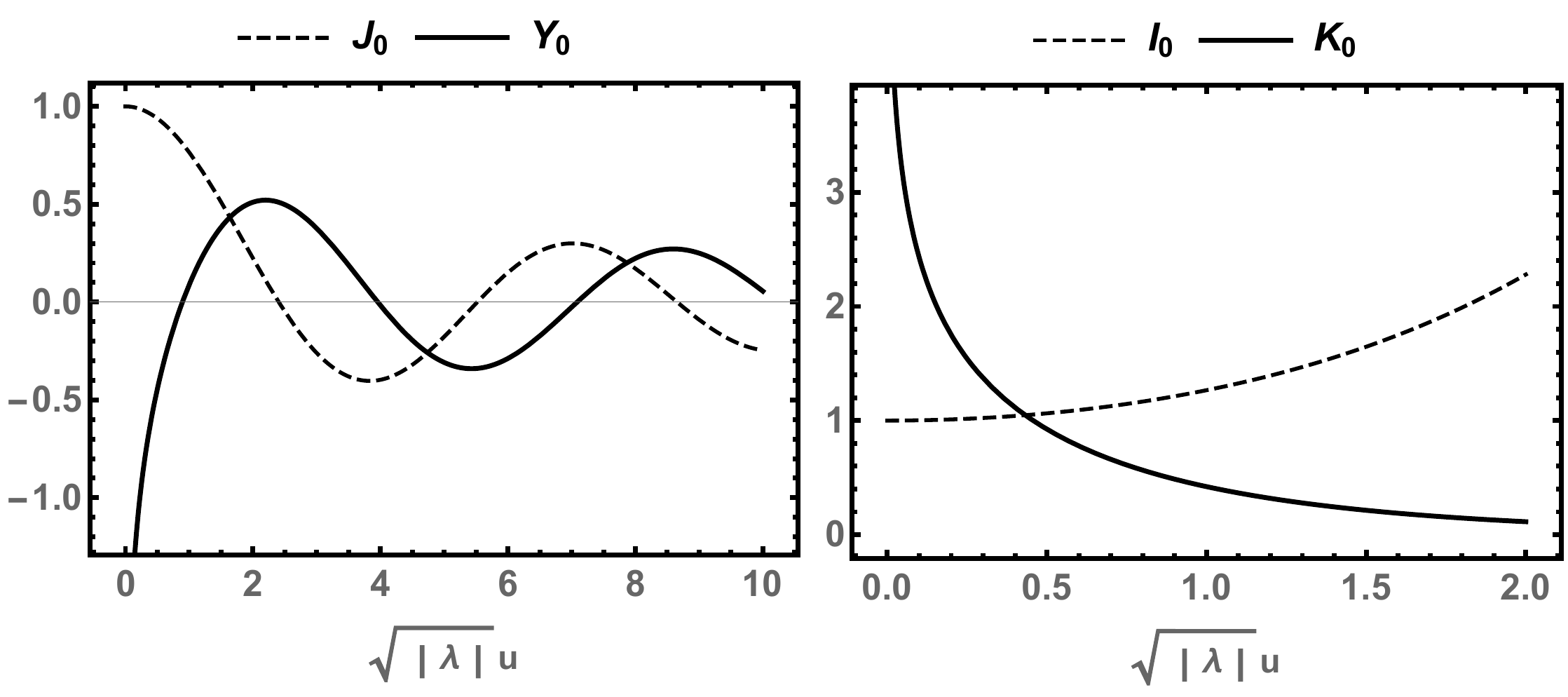}
    \caption{Curves leading to revolution surfaces with coordinate finite-type  minimal normal $\mathbf{N}_m$ according to Theorem \ref{Thr::FiniteTypeTopViewGaussMapH}. (Left) Zero order Bessel functions of the first and second type $J_0$ and $Y_0$, respectively, whose  corresponding revolution surfaces have $\lambda>0$; (Right) Zero order Bessel functions of the third and fourth type $I_0$ and $K_0$, respectively, whose  corresponding revolution surfaces have $\lambda<0$.}
    \label{fig:BesselFunctions}
\end{figure}

\begin{theorem}\label{Thr::FiniteTypeTopViewGaussMapH}
Let $M_c^2$ be a helicoidal surface with coordinate finite-type minimal Gauss map $\mathbf{N}_m$. Then, $M_c^2$ belongs to one of the following families: 
\begin{enumerate}[(1)]
    \item If $c\not=0$, then $\lambda_1=\lambda_2=0$ and
    \begin{equation}
        z(u)=z_0+z_1u^2+z_2\ln u,
    \end{equation}
    where $z_i$ is a constant $(i=0,1,2)$. (See Fig. \ref{fig:CMChelSurf}.)
\item If $c=0$, we have one of the following three cases below
\begin{enumerate}[(a)]  
\item If $\lambda_1=\lambda_2=0$, then 
\begin{equation}
        z(u)=z_0+z_1u^2+z_2\ln u,
    \end{equation}
where $z_i$ is a constant $(i=0,1,2)$.
\item If $\lambda_1=\lambda_2=\lambda\not=0$, then
\begin{equation} \label{hsz}
    z(u) = \left\{
    \begin{array}{ccc}
        z_0+z_1 J_0(\sqrt{\lambda}\,u)+z_2 Y_0(\sqrt{\lambda}\,u) & \mbox{ if } &  \lambda>0\\[4pt]
        z_0+z_1 I_0(\sqrt{-\lambda}\,u)+z_2 K_0(\sqrt{-\lambda}\,u) & \mbox{ if }& \lambda<0\\ 
    \end{array}
    \right., 
\end{equation}
where $z_i$ is a constant $(i=0,1,2)$ and $J_0$, $Y_0$, $I_0$, and $K_0$ are the zero order Bessel functions of the first, second, third, and fourth type, respectively. (See Fig. \ref{fig:BesselFunctions}.)
\item If $\lambda_1\not=\lambda_2$, then $z(u)$ is a constant function. 
\end{enumerate}
\end{enumerate}
\end{theorem}

\begin{proof}

\textit{Case (1):} From Eq. \eqref{Eq::EigenvProbLamb1HSurf}, it is immediate to see that if $c\not=0$, then $\lambda_1=\lambda_2=0$. Therefore, the corresponding solution for $z$ is 
\begin{equation}
z(u) = z_0+z_1u^2+z_2 \ln u
\end{equation}
for some constants $z_0,z_1,z_2$. (See Fig. \ref{fig:CMChelSurf}.)

Now, let us assume that $c=0$. We have to consider the three possibilities for the values of $\lambda_1$ and $\lambda_2$ as follows:

\textit{Case (2.a):} If $\lambda_1=\lambda_2=0$, then the solution is the same as in case (1).

\textit{Case (2.b):} If $\lambda_1,\lambda_2= \lambda \not=0$, then the eigenvalue problems become
\begin{equation}\label{Eq::EigenvProbLamb1and2NotZeroHSurf}
u^2 z'''+uz''-(1-\lambda u^2)z^{\prime}=0.
\end{equation}
Now, by defining $f=z'$, the above equation turns into the following ODE
\begin{equation}\label{Eq::EigenvProbLamb1and2NotZeroHSurfb}
u^2 f''+u f'-(1-\lambda u^2)f=0.
\end{equation}
Now, to solve this equation we consider on two cases according to the sign of $\lambda$ as follows:
\begin{itemize}
\item First case, let $\lambda >0$. Then, by taking  $v(u)=\sqrt{\lambda}\,u$ in  Eq. \eqref{Eq::EigenvProbLamb1and2NotZeroHSurfb}, we are led to the Bessel ODE of first order
\begin{equation*}
    v^2 f''(v)+v f'(v)-(1-v^2)f(v)=0,
\end{equation*}
whose solution is \cite{Olver1972}
\begin{equation}
    f(v)=c_1 J_1(v)+c_2 Y_1(v).
\end{equation}
By integrating the first and the second kind Bessel functions of order 1 and by also considering $v(u)=\sqrt{\lambda}\,u$ and $z'=f$, we get the solution of Eq. (\ref{Eq::EigenvProbLamb1and2NotZeroHSurf}) as
$$z(u)=z_0+z_1 J_0(\sqrt{\lambda}\,u)+z_2 Y_0(\sqrt{\lambda}\,u),$$ where $z_i=\mbox{constant}$.

\item Now, let $\lambda <0$. As in the previous case,  by taking  $v(u)=\sqrt{-\lambda}\,u$ in  \eqref{Eq::EigenvProbLamb1and2NotZeroHSurfb}, we are led to the modified Bessel ODE of first order
\begin{equation*}
    v^2 f''(v)+v f'(v)-(1+v^2)f(v)=0,
\end{equation*}
whose solution is \cite{Olver1972}
\begin{equation*}
    f(v)=c_1 I_1(v)+c_2 K_1(v).
\end{equation*}
By considering $f=z'$, $v(u)=\sqrt{-\lambda}\,u$ and the Bessel functions, we get the solution of Eq. (\ref{Eq::EigenvProbLamb1and2NotZeroHSurf}) as 
$z(u)=z_0+z_1 I_0(\sqrt{-\lambda}\,u)+z_2 K_0(\sqrt{-\lambda}\,u)$, where $z_i=\mbox{constant}$.
\end{itemize}

\textit{Case (2.c):} If $\lambda_1\not=\lambda_2$, then subtracting $u^2z'''+uz''-(1-\lambda_1u^2)z'=0$ from $u^2z'''+uz''-(1-\lambda_2u^2)z'=0$ gives
\begin{equation}
(\lambda_1-\lambda_2) u^2 z^{\prime}=0.
\end{equation}
Hence, since $\lambda_1-\lambda_2\not=0$, we get $z(u)=z_0$ constant.
\end{proof}

\subsection{Helicoidal surfaces with coordinate finite-type minimal normal with prescribed boundary conditions}

Notice that if the minimal normal is harmonic, $\Delta \mathbf{N}_m=0$, then it follows from Eq. \eqref{eq::KandHhelicoidalSurfaces} that the corresponding helicoidal surface has constant mean curvature $H=2z_1$, while the Gaussian curvature is $K=4z_1^2-\frac{1}{u^4}(c^2+z_2^2)$. Then, by approaching the screw axis, i.e., $u\to0$, we have $K\sim -\frac{1}{u^4}\to-\infty$ while away from it, i.e., $u\gg1$, $K\sim 4z_1^2$. Asymptotically, we have $H^2-K\sim 0$ for $u\gg1$ and consequently, $M_c^2$ should behave as a totally umbilical surface: as a plane if $M_c^2$ is minimal, i.e., if $z_1=0$, or as a sphere of parabolic type if otherwise. In fact, in terms of the position vector $\mathbf{R}(u,t)=(u\cos t,u\sin t,z_0+z_1u^2+z_2\ln u+ct)$, 
$$u\gg1,\,t\in(0,2\pi)\Rightarrow M_c^2\sim\{(x,y,z):z=\frac{1}{2p}(x^2+y^2)\},\mbox{ where }p=\frac{1}{H}.$$

On the other hand, assuming $\lambda_1=\lambda_2=\lambda\not=0$, i.e., $\mathbf{N}_m$ is not harmonic, we must have $c=0$ and we may use the known expressions for the asymptotic behavior of Bessel functions \cite{Olver1972} to deduce that near the revolution axis, \linebreak $0<u\ll1$, we have $z(u)\sim w_0+w_1\ln(\sqrt{\vert\lambda\vert}\,u)$ for some constants $w_0,w_1$: $w_1=2z_2/\pi$ if $\lambda>0$ and $w_1=-z_2$ if $\lambda<0$. Finally, far from the revolution axis, $u\gg1$, we have (See Fig. \ref{fig:BesselFunctions})
$$z(u)\sim\left\{
    \begin{array}{ccc}
        z_0+z_1\sqrt{\frac{2}{\pi\sqrt{\lambda}u}}\cos(\sqrt{\lambda}u-\frac{\pi}{4})+z_2\sqrt{\frac{2}{\pi\sqrt{\lambda}u}}\sin(\sqrt{\lambda}u-\frac{\pi}{4}) & \mbox{ if } &  \lambda>0\\[5pt]
        z_0+z_1\frac{1}{\sqrt{2\pi\sqrt{-\lambda}u}}\mathrm{e}^{\sqrt{-\lambda}u}+z_2\sqrt{\frac{\pi}{2\sqrt{-\lambda}u}}\mathrm{e}^{-\sqrt{-\lambda}u} & \mbox{ if }& \lambda<0\\ 
    \end{array}
    \right..$$
    
From the expressions above, we conclude that
\begin{proposition}
Let $M_c^2$ be a helicoidal surface with coordinate finite-type minimal Gauss map $\mathbf{N}_m$. Assume that $M_c^2$ is not a plane. We have
\begin{enumerate}[(1)]
    \item If $M_c^2$ is bounded near the screw axis, $0<u\ll1$, then  $\lambda_1=\lambda_2=\lambda$ and
    \begin{equation}\label{eq::HelMinNormalBoundedNearAxis}
    z(u) = \left\{
    \begin{array}{ccc}
        z_0+z_1 J_0(\sqrt{\lambda}\,u) & \mbox{ if } &  \lambda>0\\[2pt]
        z_0+z_1u^2 & \mbox{ if } &  \lambda=0\\[2pt]
        z_0+z_1 I_0(\sqrt{-\lambda}\,u) & \mbox{ if }& \lambda<0\\ 
    \end{array}
    \right., 
\end{equation}
    where $z_i$ is a constant $(i=0,1)$ and $J_0$ and $I_0$ are the zero order Bessel functions of the first and third type, respectively. (If $\lambda\not=0$, then $c=0$.)
\item If $M_c^2$ is bounded at infinity, $u\gg1$, then $\lambda_1=\lambda_2=\lambda$ and
\begin{equation}
    z(u) = \left\{
    \begin{array}{ccc}
        z_0+z_1 J_0(\sqrt{\lambda}\,u)+z_2 Y_0(\sqrt{\lambda}\,u) & \mbox{ if } &  \lambda>0\\[2pt]
        z_0+z_2 K_0(\sqrt{-\lambda}\,u) & \mbox{ if }& \lambda<0\\ 
    \end{array}
    \right., 
\end{equation}
where $z_i$ is a constant $(i=0,1,2)$ and $J_0$, $Y_0$, and $K_0$ are the zero order Bessel functions of the first, second, and fourth type, respectively.
\item If $M_c^2$ is bounded near the axis and at infinity, then $\lambda_1=\lambda_2=\lambda>0$ and
\begin{equation}
    z(u) = z_0+z_1 J_0(\sqrt{\lambda}\,u), 
\end{equation}
where $z_i$ is a constant $(i=0,1)$ and $J_0$ is the zero order Bessel function of the first type.
\end{enumerate}
\end{proposition}

Other common boundary conditions to impose are homogeneous or periodic conditions, i.e., $z(a)=0=z(a+L)$ or $z(a)=z(a+nL)$, $\forall\,n\in\mathbb{N}$, for given parameters $a\geq0$ and $L>0$, respectively. We are not going to exactly solve these boundary conditions problems here, but notice that from the general solutions in terms of $\{1,u^2,\ln u\}$ when $\lambda=0$ or in terms of Bessel functions $J_0,Y_0,I_0$, and $K_0$, when $\lambda\not=0$, we can see that in order for $z(u)$ to satisfy the given boundary conditions mentioned above we should necessarily have $\lambda>0$. This also implies that $c=0$ and, consequently, the associated surface is a surface of (Euclidean) revolution around the $z$-axis.

To finish the analysis of helicoidal surfaces with coordinate finite-type minimal normal, let us impose mixed boundary conditions. 

\begin{proposition}
Let $M_c^2$ be a helicoidal surface with coordinate finite-type minimal Gauss map $\mathbf{N}_m$ and generating curve $\alpha(u)=(u,0,z(u))$. Assume that $M_c^2$ is not a plane. If $M_c^2$ is bounded near the screw axis, $0<u\ll1$, and $z(L)=0$ for a given $L>0$, then up to translations along the $z$-direction
    \begin{equation}
    z(u) = z_1 J_0(\sqrt{\lambda_n}\,u),\,\lambda_n = \frac{u_n^2}{L^2}, 
\end{equation}
    where $z_1$ is a constant, $J_0$ is the zero order Bessel function of the first kind, and $0=u_0<u_1<\dots<u_n\stackrel{n}{\to}\infty$ are the zeros of $J_0$.
\end{proposition}
\begin{proof}
Since we are demanding the solution $z(u)$ to be bounded near $u=0$, the function $z(u)$ has the form given in Eq. \eqref{eq::HelMinNormalBoundedNearAxis}. For simplicity, we may set $z_0=0$ and, therefore, we should also assume that $z_1\not=0$. (Geometrically, $z_0$ is associated with a translation of $M_c^2$ along the $z$-direction.) Since we should have $z(L)=0$, we see that $\lambda>0$. Finally, this boundary condition leads to
\begin{equation}
    z(L)=z_1J_0(\sqrt{\lambda}L)=0\Leftrightarrow \sqrt{\lambda}L=u_n,
\end{equation}
where $0=u_0<u_1<\dots<u_n\stackrel{n}{\to}\infty$ are the zeros of $J_0$, see \cite{Olver1972}.
\end{proof}

\subsection{Helicoidal surfaces with coordinate finite-type parabolic Gauss map}

On Theorem \ref{Thr::FiniteTypeTopViewGaussMapH}, we investigated the eigenvalue problem for $\mathbf{N}_m$. The problem for the third coordinate of $\mathbf{N}_m$ is trivial, but the solutions for the first two coordinates can be applied to the parabolic Gauss map $\mathbf{G}$. The strategy now consists in checking the compatibility of the solution for the first two coordinates $G^1$ and $G^2$ with the last one $G^3$. 

The Laplacian of a helicoidal surface given in Eq. (\ref{laph}), when applied to $G^3$, in the last coordinate in Eq. (\ref{pgmh}), gives
\begin{eqnarray}
\Delta_g G^3 & = & - \frac{2c^2}{u^4}-\frac{z'z''}{u}-(z'z'')^{\prime}. \label{eq::LapCoord3ParabNormalOfHSurf}
\end{eqnarray}
Notice that for $\lambda_1=0$ or $\lambda_2=0$, but not $\lambda_1=\lambda_2=0$, from \eqref{Eq::EigenvProbLamb1HSurf} we necessarily have $c=0$ and $z(u)=z_0$ constant. Then, the parabolic Gauss map \eqref{pgmh} is $\mathbf{G}=(0,0,\frac{1}{2})$ and it follows that in order to satisfy $-\Delta (G^1,G^2,G^3)=(\lambda_1G^1,\lambda_2G^2,\lambda_3G^3)$ we must have $\lambda_3=0$. (Here, the arbitrariness of $\lambda_1$ or $\lambda_2$ comes from the fact that the corresponding coordinates of $\mathbf{G}$ vanish identically.) In the following theorem we shall only consider the cases where $\lambda_1=\lambda_2$.

\begin{theorem} \label{Thr::FiniteTypeParabGaussMapHSurf}
Let $M_c^2$ be a helicoidal surface given by Eq. \eqref{hs1} whose top-view projection of the parabolic Gauss map $\mathbf{G}$ is of coordinate finite-type as described in Theorem \ref{Thr::FiniteTypeTopViewGaussMapH} with $\lambda_1=\lambda_2$. In addition, if the third coordinate of $\mathbf{G}$ is an eigenfunction of the Laplacian, $-\Delta_gG^3=\lambda_3G^3$, then 
$M_c^2$ is a piece of a plane.
\end{theorem}
\begin{proof} Now, notice that from \eqref{eq::LapCoord3ParabNormalOfHSurf}, we have 
\begin{eqnarray}
-\Delta G^3 & = & \frac{z'z''}{u}+(z'')^2+z'z'''+\frac{2c^2}{u^4} = \frac{z'z''+uz''^2+uz'z'''}{u}+\frac{2c^2}{u^4}\nonumber\\
& = & \frac{(uz'z'')'}{u}+\frac{2c^2}{u^4}= \frac{1}{2u}[u(z'^2)']'+\frac{2c^2}{u^4}.
\end{eqnarray}
By considering this in \eqref{pgmh}, the eigenvalue problem for $G^3$, i.e., $-\Delta G^3= \lambda_3 G^3$, can be rewritten in a more convenient form as follows
\begin{equation} \label{evcpgmh}
   \frac{1}{2u}[u(z'^2)']'+\frac{2c^2}{u^4}= \frac{\lambda_3 }{2}\Big(1-\frac{c^2}{u^2}-{z'}^2\Big).
\end{equation}
Now, defining $g=\frac{1}{2}(z'^2-1)$, we can rewrite this eigenvalue problem as
\begin{equation}\label{Eq::EffectiveEDO3rdCoordParNormalHelicoidal} 
    \frac{ug''+g'}{u}+\frac{2c^2}{u^4}=-\lambda_3g-\lambda_3\frac{c^2}{2u^2}\Rightarrow -ug''-g'-\lambda_3ug=\frac{\lambda_3c^2}{2u}+\frac{2c^2}{u^3}.
\end{equation}

We now divide the proof in two main cases: (1) when $\lambda=\lambda_1=\lambda_2=0$; and (2) $\lambda=\lambda_1=\lambda_2\not=0$.

\textit{Case (1)}: If $\lambda_1=\lambda_2=0$, then the solution of the second equation in Eq. \eqref{Eq::EigenvProbLamb1HSurf} is $z(u)=z_0+z_1u^2+z_2\ln u$. By considering this solution and $g=\frac{1}{2}(z'^2-1)$ in Eq. \eqref{Eq::EffectiveEDO3rdCoordParNormalHelicoidal} yields 
\begin{equation}
    ug''+g'+u\lambda_3g=2 \lambda_3u^3 z_1^2-u \left(\frac{\lambda_3}{2}-8z_1^2-2 \lambda_3z_1z_2\right)+\frac{\lambda_3z_2^2}{2 u}+\frac{2z_2^2}{u^3}.
\end{equation}
Comparison with the right-hand side of Eq. \eqref{Eq::EffectiveEDO3rdCoordParNormalHelicoidal} leads to 
\begin{equation}
    \left\{
    \begin{array}{c}
         \lambda_3z_1^2=0  \\
     \frac{\lambda_3}{2}-8z_1^2-2 \lambda_3z_1z_2=0 \\
     c^2+z_2^2=0\\
    \end{array}
    \right..
\end{equation}
From the last expression, we deduce that we must have $z_2=c=0$. Using this in the first and second expressions leads to $\lambda_3z_1^2=0=\lambda_3-16z_1^2$, from which we conclude that $\lambda_3=z_1=0$. In conclusion, the eigenproblems for $G^1$ and $G^2$ described in Cases (1) and (2.a) of Theorem \ref{Thr::FiniteTypeTopViewGaussMapH} put together with the problem $G^3$ lead to $\lambda_3=0$ and $z(u)=z_0$.

\textit{Case (2)}: If $\lambda_1=\lambda_2= \lambda \neq 0$, then from \eqref{Eq::EigenvProbLamb1HSurf} we have $c=0$. By considering this result in \eqref{Eq::EffectiveEDO3rdCoordParNormalHelicoidal}, we get
\begin{equation}\label{eq::Bessel3rdCoordfHelicoidal}
 ug''+g'+\lambda_3ug=0.
\end{equation}
If $\lambda_3=0$, then we must have
$ug''+g'=(ug')'=0$ whose solution is $g=a_2\ln u+a_1$. This solution, however, is only compatible with a general solution in terms of Bessel functions as in Case (2.b) of Theorem \ref{Thr::FiniteTypeTopViewGaussMapH} if $z_1=z_2=0$ and $a_1=a_2=0$. In other words, $z(u)=z_0$.

Now, if $\lambda_3\not=0$, we must solve Eq. \eqref{eq::Bessel3rdCoordfHelicoidal}, where $g=\frac{1}{2}(z'^2-1)$, and compare the corresponding solution for $z(u)$ with that of Case (2.b) of Theorem \ref{Thr::FiniteTypeTopViewGaussMapH}. Adopting the coordinate change $v=v(u)=\sqrt{\vert\lambda_3\vert}u$, the differential equation for $g=g(v)$ is
\begin{equation}
    v^2g''(v)+vg'(v)\pm v^2g(v)=0.
\end{equation}
Therefore, $g$ is a combination of the Bessel functions $\{J_0(\sqrt{\lambda_3}\,u),Y_0(\sqrt{\lambda_3}\,u)\}$ if $\lambda_3>0$ or $\{I_0(\sqrt{-\lambda_3}\,u),K_0(\sqrt{-\lambda_3}\,u)\}$ if $\lambda_3<0$. Notice, in addition, that $\lambda=\lambda_1=\lambda_2$ and $\lambda_3$ should have the same sign, otherwise we would have solutions for $z(u)$ involving functions of distinct types.

First, let us assume that $\lambda,\lambda_3>0$. From the solution for $z(u)$ given in \eqref{hsz} in terms of $\lambda>0$, we deduce
\begin{equation}\label{eq::3rdParGaussMapzprime}
    z'^2=z_1^2\lambda J_1^2(\sqrt{\lambda}u)+2z_1z_2\lambda J_1(\sqrt{\lambda}u)Y_1(\sqrt{\lambda}u)+z_2^2\lambda Y_1^2(\sqrt{\lambda}u).
\end{equation}
If there were values for $\lambda$ and $\lambda_3$ leading to compatible solutions, then we would be able to write the above expression in terms of the Bessel functions $J_0$ and $Y_0$ since we need that $g=(z'^2-1)/2$. However, this is not possible. For example, $J_1^2$ necessarily involves Bessel functions of other orders in addition to $J_0$. Indeed, from $1=J_0^2(u)+2\sum_{n=1}^{\infty}J_n^2(u)$ and $J_0(2u)=J_0^2(u)+2\sum_{n=1}^{\infty}(-1)^nJ_n^2(u)$ \cite{Olver1972}, we can write
$J_1^2(u)=\frac{1}{4}-\frac{1}{4}J_0(2u)-\sum_{n=1}^{\infty}J_{2n+1}^2(u)$. (The corresponding expressions for $Y_1(u)J_1(u)$ and $Y_1^2(u)$ involve products $J_n\,Y_n$ of higher order Bessel functions in addition to $J_0$ and $Y_0$.) An alternative way to establish the incompatibility of solutions is by investigating their asymptotic behavior: from $J_1(u)\sim \sqrt{\frac{2}{\pi u}}\cos(u-\frac{\pi}{2}-\frac{\pi}{4})$ and $Y_1(u)\sim \sqrt{\frac{2}{\pi u}}\sin(u-\frac{\pi}{2}-\frac{\pi}{4})$, we conclude that Eq. \eqref{eq::3rdParGaussMapzprime} decays to zero as $1/u$, while the expression for $z'^2$ from the solution for the equation in $g$ decays as $1/\sqrt{u}$.

Finally, a similar reasoning also applies for the case where $\lambda,\lambda_3<0$ by using the corresponding identities and properties for $I_n$ and $K_n$.
\end{proof}

\section{Parabolic Revolution Surfaces with Coordinate Finite-type Gauss Map}
\label{S:ClassificationParabRev}

Since the top view projections of the parabolic and minimal Gauss maps given in Sect. \ref{Prs} are the same, we will proceed as in the study of helicoidal surfaces. In other words, we first investigate the eigenvalue problem for the minimal normal. Later, these solutions can be used to fix the first two coordinates of the parabolic normal. Finally, it remains to analyze the last coordinate of $\mathbf{G}$. The strategy then consists in checking the compatibility of the known solutions for the first and second coordinates with the eigenvalue problem of the last one.

\subsection{Parabolic revolution surfaces with coordinate finite-type minimal normal}

Since the Laplacian of the minimal normal from \eqref{mnParRev} and \eqref{eq::LapParabRevSurf} is 
\begin{equation}
\Delta_g\mathbf{N}_m=\frac{a^2+b^2}{b^2}(-z''',\frac{a}{b}z''',0),
\end{equation}
the corresponding eigenvalue problems, $-\Delta_gN_m^i=\lambda_iN_m^i$, become
\begin{equation}
-\frac{a^2+b^2}{b^2}z'''=\lambda_1(c_1 t+z')\mbox{ and }-\frac{a(a^2+b^2)}{b^3}z'''=\lambda_2(\frac{a z'-c-c_1u}{b}-c_2t).
\end{equation}
Now, using that $\{1,t\}$ is a set of linearly independent functions, we have the following sets of equations for $\lambda_1$ and $\lambda_2$:
\begin{equation}\label{Eq::EigenvProbLamb1ParabRevSurf}
\left\{
\begin{array}{c}
\lambda_1 c_1 = 0\\[2pt]
(a^2+b^2)z'''+\lambda_1b^2z'=0\\
\end{array}
\right.
\end{equation}
and
\begin{equation}\label{Eq::EigenvProbLamb2ParabRevSurf}
\left\{
\begin{array}{c}
\lambda_2 c_2 = 0\\[2pt]
a(a^2+b^2)z'''+\lambda_2b^2(az'-c-c_1u)=0\\
\end{array}
\right..
\end{equation}
Analyzing all possibilities, we have the following classification of parabolic revolution surfaces whose minimal normal $\mathbf{N}_m$ is of coordinate finite-type.

\begin{figure}[t]
    \centering
    \includegraphics[width=0.9\linewidth]{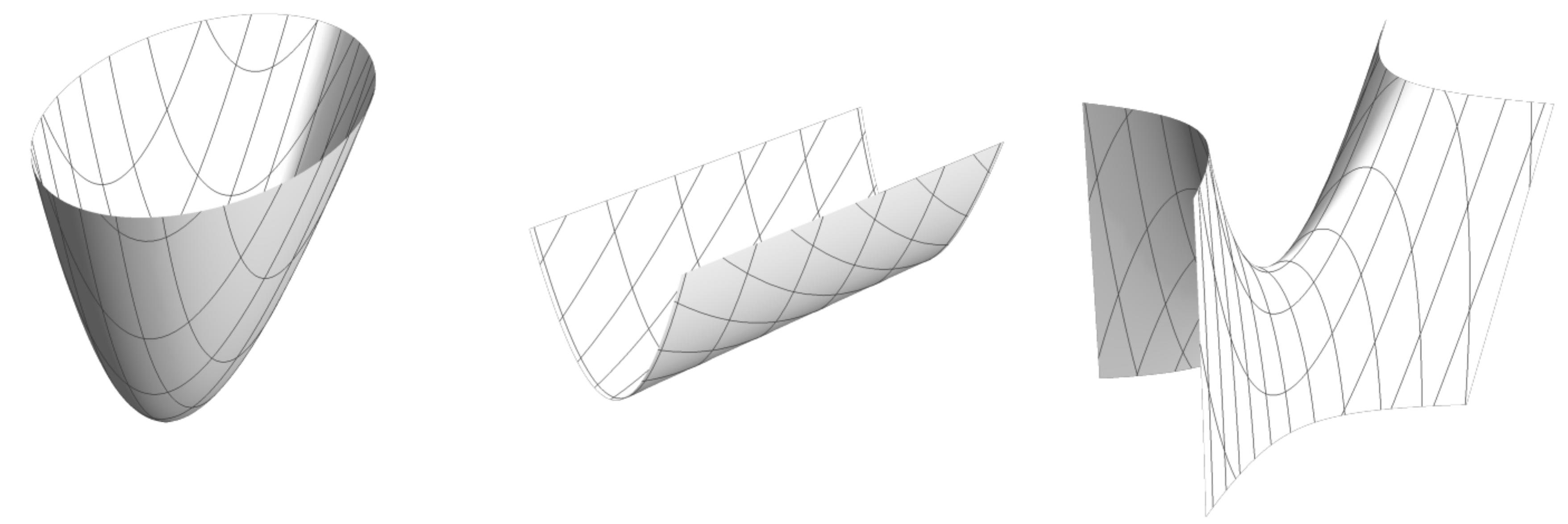}
    \caption{Parabolic revolution surfaces with harmonic minimal Gauss map $\mathbf{N}_m$ according to Theorem \ref{Thr::FiniteTypeTopViewGaussMap}. These surfaces have constant mean curvature and have an implicit equation $z+z_0=z_2x^2+2\alpha xy+\beta y^2+z_1 x+\gamma y$, where $\alpha=(c_1-2az_2)/2b$, $\beta=(2a^2z_2-ac_1+bc_2)/2b^2$, and $\gamma=(c-az_1)/b$. The surfaces are (Left) Elliptic paraboloids if $z_2\beta>\alpha^2$; (Center)  Parabolic cylinders if $z_2\beta=\alpha^2$; and (Right) Hyperbolic paraboloids if $z_2\beta<\alpha^2$. }
    \label{fig:CMCparabRevSurf}
\end{figure}

\begin{theorem}\label{Thr::FiniteTypeTopViewGaussMap}
Let $ M^2_{(a,b,c,c_1,c_2)}$ be a parabolic revolution surface with generating curve $\alpha(u)=(u,0,z(u))$ and such that the minimal normal $\mathbf{N}_m$ is of coordinate finite-type, $-\Delta_g\mathbf{N}_m=(\lambda_1N_m^1,\lambda_2N_m^2,0).$ Then, $ M^2_{(a,b,c,c_1,c_2)}$ belongs to one of the following families:   
\begin{enumerate}[(1)]  
\item If $\lambda_1=\lambda_2=0$, then $z(u)=z_2u^2+z_1u+z_0$, where $z_i$ is constant and $(a,b,c,c_1,c_2)\in\{c_1\not=0\mbox{ or }z(u)\not=\mbox{const.}\}\cap\{2az_2\not=c_1\mbox{ or }az_1\not=c\}$. (See Fig. \ref{fig:CMCparabRevSurf}.)
\item If $\lambda_1=0$ and $\lambda_2\not=0$,
then either
 \begin{enumerate}[(a)]
  \item $(a,b,c,c_1,c_2)=(0,b,0,0,0)$ and $z(u)=z_2u^2+z_1u+z_0$,
  or
  \item $(a,b,c,c_1,c_2)=(a\not=0,b,c,c_1,0)$ and $z(u)=\frac{c_1}{2a}u^2+\frac{c}{a}u+z_0$.
 \end{enumerate}
\item If $\lambda_1\not=0$ and $\lambda_2=0$, then $(a,b,c,c_1,c_2)=(a,b,c,0,c_2)$ and $z(u)=z_0$. 
\item If $\lambda_1,\lambda_2\not=0$, then either
  \begin{enumerate}[(a)]
  \item $(a,b,c,c_1,c_2)=(0,b,0,0,0)$ and 
     \begin{equation}
       z(u)=\left\{
       \begin{array}{cc}
        z_0+z_1\cos(\sqrt{\lambda_1}\,u)+z_2\sin(\sqrt{\lambda_1}\,u), &\mbox{ if }\lambda_1>0\\[2pt]
        z_0+z_1\cosh(\sqrt{-\lambda_1}\,u)+z_2\sinh(\sqrt{-\lambda_1}\,u), &\mbox{ if }\lambda_1<0\\
       \end{array}
        \right.,
      \end{equation}
  or
  \item $(a,b,c,c_1,c_2)=(a\not=0,b,0,0,0)$, $\lambda_1=\lambda_2=\lambda$, and 
  \begin{equation}
       z(u)=\left\{
       \begin{array}{cc}
        z_0+z_1\cos(\sqrt{\Lambda}\,u)+z_2\sin(\sqrt{\Lambda}\,u), &\mbox{ if }\lambda>0\\[2pt]
        z_0+z_1\cosh(\sqrt{-\Lambda}\,u)+z_2\sinh(\sqrt{-\Lambda}\,u), &\mbox{ if }\lambda<0\\
       \end{array}
        \right.,
      \end{equation}
  where $\Lambda=\frac{\lambda b^2}{a^2+b^2}$.
 \end{enumerate} 
\end{enumerate}
\end{theorem}

\begin{proof}
\textit{Case (1):} If $\lambda_1=\lambda_2=0$, then Eqs. \eqref{Eq::EigenvProbLamb1ParabRevSurf} and  \eqref{Eq::EigenvProbLamb2ParabRevSurf} lead to the same solution for $z$:
\begin{equation}
z(u) = z_0+z_1u+z_2u^2.
\end{equation}
In order to avoid a trivial coordinate eigenfunction, i.e., $N_m^i\equiv0$ (otherwise, $\lambda_i$ could be arbitrary), we have to impose $c_1\not=0$ or $z'(u)\not=0$, for the first coordinate, and $2az_2\not=c_1$ or $az_1\not=c$, for the second coordinate. 

\textit{Case (2):} If $\lambda_1=0$ but $\lambda_2\not=0$, then from the first expression of Eq. (\ref{Eq::EigenvProbLamb2ParabRevSurf}), we have $c_2=0$ and from the second expression of Eq. (\ref{Eq::EigenvProbLamb1ParabRevSurf}), we have $z(u)=z_0+z_1u+z_2u^2$. Using this information in the second expression of Eq. (\ref{Eq::EigenvProbLamb2ParabRevSurf}) gives
\begin{equation}
(2az_2-c_1)u+(az_1-c)=0.
\end{equation}
If $a=0$, then $c=c_1=0$ ($\{1,u\}$ is linearly independent) and $z(u)$ can be any quadratic polynomial. On the other hand, if $a\not=0$, then
\begin{equation}
z(u) = z_0+\frac{c}{a}u+\frac{c_1}{2a}u^2.
\end{equation}

\textit{Case (3):} If $\lambda_2=0$ but $\lambda_1\not=0$, then from the first expression of Eq. (\ref{Eq::EigenvProbLamb1ParabRevSurf}), we have $c_1=0$ and from the second expression of Eq. (\ref{Eq::EigenvProbLamb2ParabRevSurf}), we have $z(u)=z_0+z_1u+z_2u^2$. Using this information in Eq. (\ref{Eq::EigenvProbLamb1ParabRevSurf}) gives $z'=0\Rightarrow z(u)=z_0$.

\textit{Case (4):} If $\lambda_1,\lambda_2\not=0$, then from the first expressions in both Eqs. (\ref{Eq::EigenvProbLamb1ParabRevSurf}) and (\ref{Eq::EigenvProbLamb2ParabRevSurf}), we have $c_1=c_2=0$. The eigenvalue problems become
\begin{equation}\label{Eq::EigenvProbLamb1and2NotZeroParRevSurf}
\left\{
\begin{array}{c}
z'''+\lambda_1\frac{b^2}{a^2+b^2}z'=0\\[3pt]
az'''+\lambda_2\frac{b^2}{a^2+b^2}(az'-c)=0\\
\end{array}
\right..
\end{equation}
Now, we have 2 sub-cases to be analyzed: $a=0$ or $a\not=0$. If $a=0$, then $\lambda_2c=0$ and, therefore, $c=0$. Finally, the equation for $\lambda_1$ gives
$$
z'''+\lambda_1z'=0 \Rightarrow
z(u)=\left\{
\begin{array}{cc}
z_0+z_1\cos(\sqrt{\lambda_1}\,u)+z_2\sin(\sqrt{\lambda_1}\,u), &\,\lambda_1>0\\[2pt]
z_0+z_1\cosh(\sqrt{-\lambda_1}\,u)+z_2\sinh(\sqrt{-\lambda_1}\,u),& \,\lambda_1<0\\
\end{array}
\right..
$$
On the other hand, if $a\not=0$, then from the first expression of Eq. (\ref{Eq::EigenvProbLamb1and2NotZeroParRevSurf}), 
\begin{equation}
z(u)=\left\{
\begin{array}{cc}
z_0+z_1\cos(\sqrt{\Lambda_1}\,u)+z_2\sin(\sqrt{\Lambda_1}\,u), &\mbox{ if }\lambda_1>0\\[2pt]
z_0+z_1\cosh(\sqrt{-\Lambda_1}\,u)+z_2\sinh(\sqrt{-\Lambda_1}\,u), &\mbox{ if }\lambda_1<0\\
\end{array}
\right.,
\end{equation}
where $\Lambda_1=\lambda_1\frac{b^2}{a^2+b^2}$. Using this in the second expression of Eq. (\ref{Eq::EigenvProbLamb1and2NotZeroParRevSurf}), we have
\begin{equation}
(\lambda_2-\lambda_1)z' = \frac{c}{a}\lambda_2.
\end{equation}
If it were $\lambda_1\not=\lambda_2$, then we would have $z(u)$ linear, what contradicts the expression for $z(u)$ as a linear combination of (hyperbolic) trigonometric functions. Therefore, we conclude that $\lambda_1=\lambda_2$, from which it follows from the expression above that $c=0$.
\end{proof}

\begin{corollary}\label{Cor::FiniteTypeTopViewGaussMap}
Let $ M^2_{(a,b,c,c_1,c_2)}$ be a parabolic revolution surface such that the minimal Gauss map $\mathbf{N}_m$ is of coordinate finite-type with eigenvalues $\lambda_1,\lambda_2$. If $\mathbf{N}_m$ is not harmonic, then $ M^2_{(a,b,c,c_1,c_2)}$ is a non-isotropic cylinder generated by 
\begin{enumerate}[(1)]
\item a parabola if $\lambda_1=0$ or $\lambda_2=0$, but not both; or
 \item a linear combination of (hyperbolic) sine and cosine functions if $\lambda_1\lambda_2\not=0$.  
\end{enumerate}
\end{corollary}

\begin{proof}
All the surfaces from the cases (2.a), (4.a), and (4.b) in Theorem \ref{Thr::FiniteTypeTopViewGaussMap} are clearly cylinders parameterized as
\begin{equation}
\mathbf{P}(u,t) = (u,0,z(u))+t(a,b,0),
\end{equation}
where $z(u)$ is either quadratic or a linear combination of (hyperbolic) trigonometric functions. In case (2.b), after employing the coordinate change $(v,t)=(u+at,t)$, the corresponding surface is the parabolic cylinder
\begin{equation}
\mathbf{P}(v,t) = (v,0,z_0+\frac{c}{a}v+\frac{c_1}{2a}v^2)+t(0,b,0).
\end{equation}
Finally, for the remaining case (3), the corresponding surface is a parabolic cylinder
\begin{equation}
\mathbf{P}(u,t) = \beta(t)+u(1,0,0),
\end{equation}
where the generating curve is the parabola $\beta(t)=t(a,b,c)+[\frac{bc_2}{2}t^2+z_0](0,0,1)$.
\end{proof}

\subsection{Parabolic revolution surfaces with coordinate finite-type minimal normal with prescribed boundary conditions}

Notice that the spectra are continuous for all the surfaces obtained in Theorem \ref{Thr::FiniteTypeTopViewGaussMap}. In fact, it is possible to find surfaces for any given value of $(\lambda_1,\lambda_2)\in\mathbb{R}^2$. (Notice that in cases (2.a), (2.b), (3), and (4.a) the solutions do not depend on one of the eigenvalues, what is explained by the fact that the corresponding coordinate of the normal field vanishes identically.) To obtain a discrete spectrum, it is necessary to impose some sort of boundary conditions on the generating curve $\alpha(u)=(u,0,z(u))$.

\begin{proposition}
Let $M_{(a,b,0,0,0)}^2$ be a parabolic revolution surface with generating curve $\alpha(u)=(u,0,z(u))$ and with coordinate finite-type  minimal Gauss map $\mathbf{N}_m$ with $\lambda_1=\lambda_2$. Let $a,L>0$ be constant, it follows that, up to translations along the isotropic direction, 
\begin{enumerate}[(1)]
    \item if we assume homogeneous boundary conditions, $z(a)=0=z(a+L)$, then
    \begin{equation}
    \left\{\begin{array}{c}
         z(u)=\zeta_0\sin\left(\sqrt{\Lambda_n}(u-a)\right)  \\
         \Lambda_n=\displaystyle\frac{\lambda_nb^2}{a^2+b^2}=\frac{\pi^2n^2}{L^2},\,n\in\{1,2,3,\dots\}
    \end{array}
    \right.,
    \end{equation}
    where $\zeta_0$ is a constant.
    \item if we assume periodic boundary conditions, $\forall k\in\mathbb{Z}$, $z(a)=z(a+kL)$, then
    \begin{equation}
    \left\{\begin{array}{c}
         z(u)=z_1\cos(\sqrt{\Lambda_n}u)+z_2\sin(\sqrt{\Lambda_n}u)  \\[2pt]
         \Lambda_n=\displaystyle\frac{\lambda_nb^2}{a^2+b^2}=\frac{4\pi^2n^2}{L^2},\,n\in\{0,1,2,\dots\} 
    \end{array}
    \right..
    \end{equation}
\end{enumerate}
\end{proposition}
\begin{proof} Without loss of generality, we may set $z_0=0$ in the general solutions from Theorem \ref{Thr::FiniteTypeTopViewGaussMap}. (Geometrically, $z_0$ is associated with a translation of the corresponding surface along the isotropic direction.)

\textit{Case (1):} From Theorem \ref{Thr::FiniteTypeTopViewGaussMap} we can see that we should have $\lambda>0$ and, therefore, $z(u)=z_1\cos(\sqrt{\Lambda}u)+z_2\sin(\sqrt{\Lambda}u)$. Now, applying the boundary conditions, we are led to the following equations
\begin{equation}
    \left\{
    \begin{array}{ccc}
    z_1c_a+z_2s_a     & = & 0 \\
    (c_ac_L-s_as_L)z_1+(s_ac_L+c_as_L)z_2 & = & 0\\
    \end{array}
    \right.,
\end{equation}
where $c_x=\cos(\sqrt{\Lambda}\,x)$ and $s_x=\sin(\sqrt{\Lambda}\,x)$. Since $(z_1,z_2)\not=(0,0)$, the above system of equations is degenerate and, then, the following determinant vanishes
\begin{equation}
    \left\vert
    \begin{array}{cc}
        c_a & s_a \\
    (c_ac_L-s_as_L)&(s_ac_L+c_as_L)\\
    \end{array}
    \right\vert=0\Rightarrow s_L=0.
\end{equation}
Finally, since $z(u)$ is a non-trivial solution, it follows that $\sin(\sqrt{\Lambda}L)=0$ and that $\sqrt{\Lambda}L$ must assume the discrete values $\sqrt{\Lambda_n}\,L=n\pi$ for $n=1,2,\dots$. Now, writing $(z_1,z_2)=(\zeta\sin\phi,\zeta\cos\phi)$, $\zeta\not=0$, we have
\begin{equation}
    0=z(a)=\zeta\sin(\sqrt{\Lambda_n}a+\phi)\Rightarrow \phi=k\pi-\sqrt{\Lambda_n}a,\,k\in\mathbb{Z}.
\end{equation}
Finally, we can rewrite the general solution as
\begin{equation}
    z(u)=\zeta\sin(\sqrt{\Lambda_n}u+\phi_{k})=\zeta\sin(\sqrt{\Lambda_n}(u-a)+k\pi)=\zeta_0\sin(\sqrt{\Lambda_n}(u-a)),
\end{equation}
where $\zeta_0=(-1)^k\zeta$.

\textit{Case (2):} As in the previous case, here the eigenvalues should be also positive. Working with the solution in its complex form, $z(u)=c_0\mathrm{e}^{\mathrm{i}\sqrt{\Lambda}\,u}$, and applying the boundary conditions implies
\begin{equation}
   \mathrm{e}^{\mathrm{i}\sqrt{\Lambda}\,a}=\mathrm{e}^{\mathrm{i}\sqrt{\Lambda}\,(a+L)}\Rightarrow \mathrm{e}^{\mathrm{i}\sqrt{\Lambda}\,L}=1. 
\end{equation}
Then, $\sqrt{\Lambda}L$ must assume the discrete values $\sqrt{\Lambda_n}L=2n\pi$, $\forall\,n=0,1,2,\dots.$
\end{proof}

\subsection{Parabolic revolution surfaces with parabolic Gauss map of coordinate finite-type}

On Theorem \ref{Thr::FiniteTypeTopViewGaussMap}, we investigated the eigenvalue problem for $\mathbf{N}_m$. The problem for the third coordinate of $\mathbf{N}_m$ is trivial, but the solutions for the first two coordinates can be applied to the parabolic Gauss map. The strategy now consists in checking the compatibility of the solution for the first coordinates $G^1$ and $G^2$ with the last one $G^3$. 

The Laplacian of a parabolic revolution surface, Eq. (\ref{eq::LapParabRevSurf}), when applied to $G^3$, Eq. (\ref{eq::3rdCoordNormalParabRevSurf}), gives
\begin{eqnarray}
\Delta G^3 & = & -\frac{\left(a^2+b^2\right)^2}{b^4}(z''^2+z'z''')+\frac{a\left(a^2+b^2\right) \left(c+c_1u\right)}{b^4}z'''+\label{eq::LapCoord3ParabNormalOfParabRevSurf}\\
& + & \frac{2a[2 b^2c_1+a(ac_1-bc_2)]}{b^4}z''-\frac{(ac_1-bc_2)^2+2 b^2c_1^2}{b^4}+ \frac{t}{b^3}(a^2+b^2)(a c_2-b c_1)z'''. \nonumber
\end{eqnarray}

The analysis now will be divided into two instances. The first theorem below refers to $\lambda_1=\lambda_2=0$ while the second refers to $\lambda_1=\lambda_2\not=0$. Notice we must assume that $\lambda_1=\lambda_2$ in order to avoid trivial eigenproblems, i.e., $G^1$ or $G^2$ identically zero.

\begin{theorem}\label{Thr::FiniteTypeParabGaussMapParabRevSurf}
Let $ M^2_{(a,b,c,c_1,c_2)}$ be a parabolic revolution surface with generating curve $\alpha(u)=(u,0,z(u))$ and whose top-view projection of the parabolic Gauss map $\mathbf{G}$ is of finite-type, as described in Theorem \ref{Thr::FiniteTypeTopViewGaussMap}, with $\lambda_1=\lambda_2=0$. In addition, if the third coordinate of $\mathbf{G}$ is a non-zero eigenfunction, then $\lambda_3=0$, $ (a,b,c,c_1,c_2)=(a,b,c,0,0)$, and $z(u)=z_0+z_1 u$.
\end{theorem}
\begin{proof}
Since $\lambda_1$ and $\lambda_2$ vanishes, we must have $z(u)=z_0+z_1u+z_2u^2$, which gives $z'=z_1+2z_2u$, $z''=2z_2$, and $z'''=0$. Noticing that the eigenvalue problem $-\Delta G^3=\lambda_3G^3$ can be written as a polynomial of degree 2 in $t$, we are led to three equations. The equations associated with $t^2$ and $t$ are
\begin{equation}\label{eq::lamb3xi3ParabRevSurf}
\left\{
\begin{array}{c}
\frac{\lambda_3}{2}(c_1^2+c_2^2) = 0\\[2pt]
\frac{\lambda_3}{b}[(ac_2-bc_1)(z_1+2z_2u)-c_2(c+c_1u)]=0\\
\end{array}
\right.,
\end{equation}
respectively. We have two sub-cases to consider, either $\lambda_3=0$ or $\lambda_3\not=0$. We are going to show that $\lambda_3$ must vanish.

If it were $\lambda_3\not=0$, then from the first expression in Eq. (\ref{eq::lamb3xi3ParabRevSurf}), we would have $c_1=c_2=0$ (the second expression would be trivially satisfied). Finally, the part of $-\Delta_gG^3=\lambda_3G^3$ depending on $t^0=1$ leads to the equation
\begin{eqnarray}
4\frac{(a^2+b^2)^2}{b^4}z_2^2 & = & \frac{\lambda_3}{2}\left(1-\frac{c^2}{b^2}+\frac{2ac}{b^2}z_1-\frac{a^2+b^2}{b^2}z_1^2\right)+\nonumber\\
& + & 2\lambda_3z_2(\frac{ac}{b^2}-\frac{a^2+b^2}{b^2}z_1)u-2\lambda_3\frac{a^2+b^2}{b^2}z_2^2u^2.\label{Eq::ProblLamb3WithC1C2zero}
\end{eqnarray}
From the coefficient in $u^2$, we deduce that $z_2=0$ or $a^2+b^2=0$. Since $b\not=0$, we conclude $z_2=0$ and, in addition, it follows that the coefficient in $u$ vanishes identically. In short, if it were $\lambda_3\not=0$, we would have $c_1=c_2=0$, $z(u)=z_0+z_1u$, and the parabolic normal would be
\begin{equation}
    \mathbf{G} = (-z_1,\frac{az_1-c}{b},\frac{1}{2}-\frac{c^2}{2b^2}+\frac{acz_1}{b^2}-\frac{(a^2+b^2)z_1^2}{2b^2}),
\end{equation}
which is a constant vector and, consequently, it is not compatible with $\lambda_3\not=0$.

Now, let us assume that $\lambda_3=0$. We have to analyze the equation
\[
 \Delta G^3=\frac{4az_2[2 b^2c_1+a(ac_1-bc_2)]}{b^4}-\frac{4z_2^2(a^2+b^2)^2}{b^4}-\frac{(ac_1-bc_2)^2+2 b^2c_1^2}{b^4}=0.
\]
Seeing it as a degree 2 polynomial in $c_1$, the corresponding discriminant $D_1$ is
\[
D_1 = -\frac{8}{b^4}\left[c_2^2 + 2(a^2 + 2 b^2)z_2^2\right] \leq 0.
\]
To guarantee $c_1\in\mathbb{R}$, we then have $D_1=0$ and, consequently, $(a^2+2b^2)z_2^2=0$ and $c_2^2=0$. Since $b\not=0$, we conclude that $c_2=z_2=0$. Consequently, $\Delta G^3$ becomes
\[
 \Delta G^3 = -\frac{a^2c_1^2+2 b^2c_1^2}{b^4}=0.
\]
Thus, $a^2c_1^2=0$ and $b^2c_1^2=0$ and, since $b\not=0$, we conclude in addition that $c_1=0$. In short, $\lambda_1=\lambda_2=0$ and $\lambda_3=0$ implies $(a,b,c,c_1,c_2)=(a,b,c,0,0)$ and $z(u)=z_1u+z_0$.
\end{proof}

\begin{theorem}\label{Thr::FiniteTypeParabGaussMapParabRevSurfNonHarmTopView}
Let $ M^2_{(a,b,c,c_1,c_2)}$ be a parabolic revolution surface with generating curve $\alpha(u)=(u,0,z(u))$ and whose top-view projection of the parabolic Gauss map $\mathbf{G}$ is of coordinate finite-type, as described in Theorem \ref{Thr::FiniteTypeTopViewGaussMap}, with $\lambda=\lambda_1=\lambda_2\not=0$. In addition, if the third coordinate of $\mathbf{G}$ is a non-zero eigenfunction, then $ M^2_{(a,b,c,c_1,c_2)}=M^2_{(a,b,0,0,0)}$ belongs to one of the following families: 
\begin{enumerate}[(1)]
    \item If $\lambda_3=0$, then  $z(u)=z_0$.
\item If $\lambda_3\not=0$, then $\lambda_3=4\lambda$ and 
  \begin{equation}
       z(u)=\left\{
       \begin{array}{cc}
        z_0+\sqrt{\frac{2}{\Lambda}}\sin(\sqrt{\Lambda}\,u+\phi_0), &\mbox{ if }\lambda>0\\[5pt]
        z_0+\sqrt{-\frac{2}{\Lambda}}\sinh(\sqrt{-\Lambda}\,u+\phi_0), &\mbox{ if }\lambda<0\\
       \end{array}
        \right.,
      \end{equation}
  where $z_0$ and $\phi_0$ are constant and $\Lambda=\lambda b^2/(a^2+b^2)$.
\end{enumerate} 
\end{theorem}

\begin{proof}
Since $(a,b,c,c_1,c_2)=(a,b,0,0,0)$, the eigenvalue problem $-\Delta_gG^3=\lambda_3G^3$ becomes
\begin{equation}
\left(\frac{a^2+b^2}{b^2}\right)^2(z''^2+z'z''')=\frac{\lambda_3}{2}(1-\frac{a^2+b^2}{b^2}z'^2).
\end{equation}

\textit{Case (1)}: If $\lambda_3=0$, then $\frac{1}{2}(z'^2)''=(z'z'')'=(z''^2+z'z''')=0$, whose general solution has the form $z(u)=\pm\frac{2}{3u_0}(u_0\,u+u_1)^{3/2}+u_2$. Unless $z(u)=z_0$, this contradicts the expression of $z(u)$ as a linear combination of (hyperbolic) trigonometric functions. 

\textit{Case (2)}: If $\lambda_3\not=0$, we have the equation
\[
\frac{1}{2}\left(\frac{a^2+b^2}{b^2}\right)^2(z'^2)''=\frac{\lambda_3}{2}(1-\frac{a^2+b^2}{b^2}z'^2) \Rightarrow w''=-\frac{\lambda_3b^2}{a^2+b^2}w,
\]
where $w=\frac{a^2+b^2}{b^2}z'^2-1$. Then, defining $\Lambda_3=\frac{b^2\lambda_3}{a^2+b^2}$, we have 
\begin{equation}
       z'^2=\left\{
       \begin{array}{cc}
        \frac{b^2}{a^2+b^2}\left[1+w_1\cos(\sqrt{\Lambda_3}\,u)+w_2\sin(\sqrt{\Lambda_3}\,u)\right], &\mbox{ if }\lambda_3>0\\[3pt]
        \frac{b^2}{a^2+b^2}\left[1+w_1\cosh(\sqrt{-\Lambda_3}\,u)+w_2\sinh(\sqrt{-\Lambda_3}\,u)\right], &\mbox{ if }\lambda_3<0\\
       \end{array}
        \right..
\end{equation}

As a first consequence, $\lambda_3$ should have the same sign as $\lambda$ since the signs of $\lambda$ and $\lambda_3$ determine whether the solution involves $\{\cos,\sin\}$ or $\{\cosh,\sinh\}$. Let us first assume that $\lambda>0$. Then, we can write
\begin{equation}
z(u) = z_0+z_1\cos(\sqrt{\Lambda}\,u)+z_2\sin(\sqrt{\Lambda}\,u)
\end{equation}
and
\begin{equation}
z'(u) = -z_1\sqrt{\Lambda}\sin(\sqrt{\Lambda}\,u)+z_2\sqrt{\Lambda}\cos(\sqrt{\Lambda}\,u).
\end{equation}
Using the identities $\cos^2x=\frac{1}{2}+\frac{1}{2}\cos2x$ and $\sin^2x=\frac{1}{2}-\frac{1}{2}\cos2x$, we have
\begin{eqnarray}
z'^2 & = & \Lambda z_1^2\sin^2(\sqrt{\Lambda}u)+\Lambda z_2^2\cos^2(\sqrt{\Lambda}u)+2z_1z_2\Lambda\sin(\sqrt{\Lambda}u)\cos(\sqrt{\Lambda}u)\nonumber\\
& = & \frac{z_1^2+z_2^2}{2}\Lambda+ \frac{z_2^2-z_1^2}{2}\Lambda\cos(2\sqrt{\Lambda}\,u)+z_1z_2\Lambda\sin(2\sqrt{\Lambda}\,u).\nonumber
\end{eqnarray}
Compatibility of the solutions demands the following relations between the parameters  $\{w_i,\lambda_3\}$ and $\{z_i,\lambda\}$,
\begin{equation}
\lambda_3=4\lambda\mbox{ and }
2 = \lambda(z_1^2+z_2^2),
\end{equation}
respectively. Finally, writing $z_2=\zeta\cos\phi_0$ and $z_1=\zeta\sin\phi_0$, we deduce from the second equation that $\zeta^2=2/\lambda$, while $\phi_0$ is an arbitrary constant. The expression for $z(u)$ follows from $\sin(x+y)=\sin x\cos y+\cos x\sin y$.

For the case $\lambda<0$, we can write
\begin{equation}
z(u) = z_0+z_1\cosh(\sqrt{-\Lambda}\,u)+z_2\sinh(\sqrt{-\Lambda}\,u)
\end{equation}
and
\begin{equation}
z'(u) = z_1\sqrt{-\Lambda}\sinh(\sqrt{-\Lambda}\,u)+z_2\sqrt{-\Lambda}\cosh(\sqrt{-\Lambda}\,u).
\end{equation}
Using that $\cosh^2x=\frac{1}{2}+\frac{1}{2}\cosh2x$ and $\sinh^2x=-\frac{1}{2}+\frac{1}{2}\cosh2x$, we have
\begin{eqnarray}
z'^2 & = & -\Lambda z_1^2\sinh^2(\sqrt{-\Lambda}u)-\Lambda z_2^2\cosh^2(\sqrt{-\Lambda}u)-2z_1z_2\Lambda\sinh(\sqrt{-\Lambda}u)\cosh(\sqrt{-\Lambda}u)\nonumber\\
& = & -\Lambda \frac{z_2^2-z_1^2}{2}-\Lambda \frac{z_2^2+z_1^2}{2}\cosh(2\sqrt{-\Lambda}\,u)-z_1z_2\Lambda\sinh(2\sqrt{-\Lambda}\,u).\nonumber
\end{eqnarray}
Compatibility of the solutions demands the following relations between the set of parameters  $\{w_i,\lambda_3\}$ and $\{z_i,\lambda\}$,
\begin{equation}
\lambda_3=4\lambda\mbox{ and }
2 = \lambda(z_1^2-z_2^2),
\end{equation}
respectively. Finally, writing $z_2=\zeta\cosh\phi_0$ and $z_1=\zeta\sinh\phi_0$, we deduce from the second equation that $\zeta^2=-2/\lambda$, while $\phi_0$ is an arbitrary constant. The expression for $z(u)$ follows from the identity $\sinh(x+y)=\sinh x\cosh y+\cosh x\sinh y$.
\end{proof}

\section{Simply Isotropic Surfaces with Harmonic Gauss Map}
\label{S:ClassificationHarmonicGaussMaps}

From Theorems \ref{Thr::FiniteTypeTopViewGaussMapH} and \ref{Thr::FiniteTypeTopViewGaussMap} we can deduce that those surfaces with harmonic minimal normal, $\Delta\mathbf{N}_m=0$, have constant isotropic mean curvature, see Figs. \ref{fig:CMChelSurf} and \ref{fig:CMCparabRevSurf}. (See \cite{daSilvaMJOU2019} for the characterization of invariant surfaces with constant isotropic mean curvature.) In this final section we show that this is valid in general. More precisely, now we address the problem of characterization those surfaces with harmonic minimal or parabolic Gauss map without the assumption that they are invariant.

Any admissible surface can be parameterized in normal form as the graph of a smooth function $f$:
\begin{equation}
    \mathbf{x}(u^1,u^2)=(u^1,u^2,f(u^1,u^2)).
\end{equation}
Then, the vectors spanning the tangent planes are $\mathbf{x}_1 = (1,0,f_1)$ and $\mathbf{x}_2 = (0,1,f_2)$, where $f_i=\partial f/\partial u^i$. The minimal and parabolic normals are
\begin{equation}
    \mathbf{N}_m = (-f_1,-f_2,1)
\mbox{ and }
    \mathbf{G} = (-f_1,-f_2,\frac{1}{2}-\frac{1}{2}(f_1^2+f_2^2)),
\end{equation}
respectively. Finally, the first and second fundamental forms are 
\begin{equation}
\mathrm{I} = (\rmd u^1)^2+(\rmd u^2)^2 \mbox{ and }    \mathrm{II} = f_{ij}\rmd u^i\rmd u^j,
\end{equation}
from which we compute the shape operator, mean and Gaussian curvatures as
\begin{equation}
    S(p) = \mbox{Hess}_pf,\, H= \frac{\Delta f}{2},\mbox{ and }K=f_{11}f_{22}-f_{12}^2,
\end{equation}
where $f_{ij}=\partial^2f/\partial u^i\partial u^j$.
\begin{proposition}
The Laplacian of the minimal and parabolic normal vector fields are given by
\begin{equation}
    \Delta \mathbf{N}_m = (-2H_1,-2H_2,0)
\mbox{ and }
    \Delta \mathbf{G} = -2\nabla H-\tr(S^2)\mathcal{N},
\end{equation}
where $H_i=\partial H/\partial u^i$, $\nabla H = H_1\mathbf{x}_1+H_2\mathbf{x}_2$, $\tr(S^2)=4H^2-2K$, and $\mathcal{N}=(0,0,1)$ is the metric isotropic normal.
\end{proposition}
\begin{proof}
Since the metric in normal form is the identity, the Laplace-Beltrami operator of $M^2$ is just the usual plane Laplacian operator $\Delta=\partial_1^2+\partial_2^2$. Then,
\begin{eqnarray}
\Delta \mathbf{N}_m & = & (-\Delta f_1,-\Delta f_2,0)= (-f_{111}-f_{122},-f_{211}-f_{222},0)\nonumber\\
& = & (-\partial_1\Delta f,-\partial_2\Delta f,0)=(-2H_1,-2H_2,0).
\end{eqnarray}
On the other hand, for the parabolic normal
\begin{eqnarray}
\Delta \mathbf{G} & = & (-\Delta f_1,-\Delta f_2,\Delta(\frac{1}{2}-\frac{f_1^2+f_2^2}{2}))\nonumber\\
& = & -(\partial_1\Delta f,\partial_2\Delta f,f_{1}\partial_1\Delta f+f_2\partial_2\Delta f+(f_{11}+f_{22})^2-2f_{11}f_{22}+2f_{12}^2)\nonumber\\
& = & -(\partial_1\,\Delta f)\mathbf{x}_1-(\partial_2\,\Delta f)\mathbf{x}_2-[(\tr\, S)^2-2\det S]\mathcal{N}\nonumber\\
& = & -2\nabla H-\tr(S^2)\mathcal{N},
\end{eqnarray}
where we used that $\tr(S^2)=(\tr\,S)^2-2\det S=4H^2-2K$.
\end{proof}

In Euclidean space, the Laplacian of the surface normal $\mathbf{N}_{\mathrm{eucl}}$ is
\begin{equation}\label{eq::LapEuclNormal}
\Delta_g \mathbf{N}_{\mathrm{eucl}}=-2\nabla H-\tr(S^2)\mathbf{N}_{\mathrm{eucl}}.    
\end{equation}
Thus, the Laplacian of the parabolic normal can be seen as the isotropic analog of this expression, but in $\mathbb{I}^3$ we have to mix together two types of normal vector fields, $\mathbf{G}$ and $\mathcal{N}$. (As a matter of fact, the isotropic analog of the known expression for the position vector of a surface in $\mathbb{E}^3$, namely $\Delta_g\mathbf{x}=2H\mathbf{N}_{\mathrm{eucl}}$, is given by $\Delta_g\mathbf{x}=2H\mathcal{N}$ \cite{SatoArXiv2018}.) On the other hand, the minimal normal allows us to characterize constant mean curvature isotropic surfaces.

\begin{theorem}\label{Thr::CharHarmonicGaussMaps}
Let $M^2\subset\mathbb{I}^3$ be an admissible  surface, then 
\begin{enumerate}[(1)]
    \item The minimal normal $\mathbf{N}_m$ is harmonic, $\Delta\mathbf{N}_m=0$, if and only if $M^2$ has constant isotropic mean curvature;
    \item The parabolic normal $\mathbf{G}$ is harmonic, $\Delta\mathbf{G}=0$, if and only if $M^2$ is a piece of a plane.
\end{enumerate}
\end{theorem}
\begin{proof}
Case (1): notice that $\Delta \mathbf{N}_m=0\Leftrightarrow H_1=0$ and $H_2=0$.

Case (2): here $\Delta \mathbf{G}=0\Leftrightarrow \nabla H=0$ and $\tr(S^2)=0$. Since $\tr(S^2)=(f_{11})^2+2(f_{12})^2+(f_{22})^2$, then $\tr(S^2)=0$ if and only if $f_{ij}=0$. This last condition is equivalent to $f(x,y)=Ax+By+C$, for some constants $A,B,C$.
\end{proof}

\section{Concluding Remarks}

In this work, we pointed to the fact that in $\mathbb{I}^3$ there are more than one meaningful way to define a Gauss map. To better understand this issue, we studied invariant surfaces with coordinate finite-type Gauss map by choosing either the minimal $\mathbf{N}_m$ or the parabolic $\mathbf{G}$ normal, Eqs. \eqref{eq::DefMinimalNormal} and \eqref{pgm}. For $\mathbf{N}_m$, this generically led us to (at least 4-parameters) families of invariant surfaces in Theorems \ref{Thr::FiniteTypeTopViewGaussMapH} and \ref{Thr::FiniteTypeTopViewGaussMap}. On the other hand, for $\mathbf{G}$, the same condition is much more restrictive: we only have planes and certain trigonometric cylinders in Theorems \ref{Thr::FiniteTypeParabGaussMapHSurf}, \ref{Thr::FiniteTypeParabGaussMapParabRevSurf}, and \ref{Thr::FiniteTypeParabGaussMapParabRevSurfNonHarmTopView}. It is worth comparing these results with their Euclidean counterparts. When applied to invariant surfaces in $\mathbb{E}^3$, the coordinate finite-type condition leads to circular cylinder and spheres \cite{BV1993,BB1992,DPVG1990} only, in contrast to what happens in isotropic space where  we may have very distinct classes of solutions depending on the choice of the Gauss map. 

It remains to attack the same problem without the hypothesis of invariance. In this respect, it would be interesting to look for examples in other classes of surfaces, such as translation \cite{AE2017,BZ2008,BKY2016,BKY2017,S14} and factorable \cite{AEE20} surfaces. Here, we were only able to  characterize non-necessarily invariant surfaces with harmonic Gauss map. This led us to constant mean curvature surfaces for $\mathbf{N}_m$ and planes for $\mathbf{G}$, Theorem \ref{Thr::CharHarmonicGaussMaps}. It is worth mentioning that it is also possible to characterize constant mean curvature surfaces in $\mathbb{E}^3$ by using their Gauss map by relaxing the eigenvalue equation in allowing it to be satisfied pointwisely \cite{KK15,T16}, see Eq. \eqref{eq::LapEuclNormal}. Thus, we may naturally ask how much can we enlarge the class of solutions by studying surfaces with pointwise finite type Gauss map in $\mathbb{I}^3$.

\section*{Acknowledgements}
The Authors would like to thank Nurettin Cenk Turgay for useful discussions. In addition, the second named Author would like to thank the financial support provided by the Morá Miriam Rozen Gerber Fellowship for Brazilian Postdocs.







\begin{thebibliography}{10}

\bibitem{AFL1992} 
L.J. Al\'ias, A. Ferr\'andez, P. Lucas, Submanifolds in pseudo-Euclidean spaces satisfying the condition $\Delta x=Ax+b$, Geom. Dedicata {\bf42} (1992) 345--354.

\bibitem{AFL1998}
L.J. Al\'ias, A. Ferr\'andez, P. Lucas, On the Gauss map of B-scrolls, Tsukuba J. Math. {\bf22} (1998) 371--377.

\bibitem{AGM1998} 
J. Arroyo, O.J.  Garay, J.J. Menc\'ia, On a family of surfaces of revolution of finite Chen-type, Kodai Math. J. {\bf21} (1998) 73--80.

\bibitem{AEE20}
M.E. Aydin, A. Erdur, M. Ergut, Affine factorable surfaces in isotropic spaces, TWMS J. Pure Appl. Math. {\bf11} (2020) 72--88.

\bibitem{AE2017} 
M.E. Aydin, M. Erg\"{u}t, Affine translation surfaces in the isotropic 3-space. Int. Electron, J. Geom. {\bf10} (2017) 21--30.

\bibitem{BV1993}
C. Baikoussis, L.  Verstraelen, On the Gauss map of helicoidal surfaces, Rend. Sem. Mat. Messina Ser. II {\bf2} (1993) 31--42 

\bibitem{BB1992} 
C. Baikoussis, D.E. Blair, On the Gauss map of ruled surfaces, Glasgow Math. J. {\bf34} (1992) 355--359.


\bibitem{BZ2008} 
M. Bekkar, H. Zoubir, Surfaces of revolution in the 3-dimensional Lorentz-Minkowski space satisfying $\Delta x_i = \lambda_i x_i$, Int. J. Contemp. Math. Sci. {\bf3} (2008) 1173--1185.

\bibitem{BKY2016} 
B. Bukcu, M.K. Karacan, D.W. Yoon, Translation surfaces in the three-dimensional simply isotropic space $\mathbb{I}^1_3$ satisfying $\Delta^{III} x_i=\lambda_i x_i$, Konuralp J. Math.
{\bf4} (2016) 275--281.

\bibitem{BKY2017} 
B. Bukcu, M.K. Karacan, D.W. Yoon, Translation surfaces of type-2 in the three-dimensional simply isotropic space $\mathbb{I}^1_3$, Bull. Korean Math. Soc. {\bf54} (2017)  953--965.

\bibitem{CKK2019} 
A. Cakmak, M.K. Karacan, S. Kiziltug, Dual surfaces defined by $z = f(u) + g(v)$ in simply isotropic 3-space $\mathbb{I}^1_3$, Commun. Korean Math. Soc. {\bf34} (2019) 267--277. 

\bibitem{C1984} 
B.-Y. Chen, Total mean curvature and submanifolds of finite type. World Scientific, New Jersey (1984).

\bibitem{C2014} 
B.-Y. Chen, Some open problems and conjectures on submanifolds of finite type: recent development, Tamkang J. Math. {\bf45} (2014) 87--108.

\bibitem{C1996} 
B.-Y. Chen, A report on submanifolds of finite type, Soochow J. Math. {\bf22} (1996) 117--337.

\bibitem{C1987} 
B.-Y. Chen, Surfaces of finite type in Euclidean 3-space, Bull. Soc. Math. Belg. Ser. B {\bf39} (1987) 243--254.

\bibitem{CMN1986} 
B.-Y. Chen, J. Morvan, T. Nore, Energy, tension and finite type maps, Kodai Math. J. {\bf9} (1986) 406--418.

\bibitem{CP1987} 
B.-Y. Chen, P.  Piccinni, Submanifolds with finite type Gauss map, Bull. Austral. Math. Soc. {\bf35} (1987) 161--186.


\bibitem{C1995} 
S.M. Choi, On the Gauss map of surfaces of revolution in a 3-dimensional Minkowski space, Tsukuba J. Math. {\bf19} (1995) 351--367.

\bibitem{CR1995} 
S.M. Choi, On the Gauss map of ruled surfaces in a 3-dimensional Minkowski space, Tsukuba J. Math. {\bf19} (1995) 285--304.

\bibitem{daSilvaTamkang}
L.C.B. da Silva, Rotation minimizing frames and spherical curves in simply isotropic and pseudo-isotropic 3-spaces, Tamkang J. Math. {\bf51} (2020) 31--52.

\bibitem{daSilvaJG2019}
L.C.B. da Silva, The geometry of Gauss map and shape operator in simply isotropic and pseudo-isotropic spaces, J. Geom. {\bf110} (2019) 31.

\bibitem{daSilvaMJOU2019}
L.C.B. da Silva, Differential geometry of invariant surfaces in simply isotropic and pseudo-isotropic spaces, Math. J. Okayama Univ. (2021), in press, arXiv:1810.00080.

\bibitem{DPV1990}
F. Dillen, J. Pas, L.  Verstraelen, On surfaces of finite type in Euclidean 3-space, Kodai Math. J. {\bf13} (1990) 10--21.

\bibitem{DPVG1990} 
F. Dillen, J. Pas, L.  Verstraelen, On the Gauss map of surfaces of revolution, Bull. Inst. Math. Acad. Sinica {\bf18} (1990) 239--246.

\bibitem{G1988}
O.J. Garay, On a certain class of finite type surfaces of revolution. Kodai Math. J. {\bf11} (1988) 25--31.

\bibitem{G1990}
O.J. Garay, An extension of Takahashi’s theorem, Geom. Dedicata {\bf34} (1990) 105--112.

\bibitem{HB2009} 
Ch.B. Hamed, M. Bekkar, Helicoidal Surfaces in the three-dimensional Lorentz-Minkowski space satisfying $\Delta x_i = \lambda_i x_i$, Int. J. Contemp. Math. Sci. {\bf4} (2009) 311--327.

\bibitem{HV1991}
T. Hasanis, T. Vlachos, Coordinate finite-type submanifolds. Geom. Dedicata {\bf37} (1991) 155--165.

\bibitem{J1996}
C. Jang, Surfaces with 1-type Gauss map, Kodai Math. J. {\bf19} (1996) 388--394.


\bibitem{KYB2016} 
M.K. Karacan, D.W. Yoon, B. Bukcu, Translation surfaces in the three-dimensional simply isotropic space $\mathbb{I}^1_3$, Int. J. Geom. Meth. Mod. Phys. {\bf13} (2016) 1650088. 

\bibitem{KYB2017}
M.K. Karacan, D.W. Yoon, B. Bukcu, Surfaces of revolution in the three-dimensional simply isotropic space $\mathbb{I}^1_3$, Asia Pac. J. Math. {\bf4} (2017) 1--10.

\bibitem{KYK2017} 
M.K. Karacan, D.W. Yoon, S. Kiziltug, Helicoidal surfaces in the three-dimensional simply isotropic space $\mathbb{I}^1_3$, Tamkang J. Math. {\bf48} (2017) 123--134.

\bibitem{KYY2017} 
M.K. Karacan, D.W. Yoon, N. Yuksel, Classification of some special types ruled surfaces in simply isotropic 3-space, Analele Universitatii de Vest, Timisoara Seria Matematica – Informatica {\bf55} (2017) 87--98.

\bibitem{KK15}
D.-S. Kim, H.K. Kim, Shape operator and Gauss map of pointwise 1-type, J. Korean Math. Soc. \textbf{52} (2015) 1337--1346.

\bibitem{Mueller1921}
E. M\"uller, Relative Minimalfl\"achen, Monatsh. Math. Phys. {\bf31} (1921) 3--19.
 
\bibitem{Olver1972} 
F.W.J. Olver, Bessel functions of integer order. In: M. Abramowitz, I.A. Stegun (eds.), Handbook of mathematical functions: with formulas, graphs and mathematical tables, pp. 355--434. Dover, New York (1972).

\bibitem{Sachs1990}
H. Sachs, Isotrope {G}eometrie des {R}aumes, Vieweg, Braunschweig/Wiesbaden (1990).
 
\bibitem{SatoArXiv2018}
Y. Sato, $d$-minimal surfaces in three-dimensional singular semi-Euclidean space $\mathbb{R}^{0,2,1}$, e-print arXiv:1809.07518

\bibitem{SB2015}
B. Senoussi, M. Bekkar, Helicoidal surfaces with $\Delta^J r = Ar$ in 3-dimensional Euclidean space, Stud. Univ. Babes-Bolyai Math. {\bf60} (2015) 437--448.

\bibitem{Simon1991}
U. Simon, A. Schwenk-Schellschmidt, H. Viesel, Introduction to the affine differential geometry of hypersurfaces, Science University of Tokyo, Tokyo (1991).

\bibitem{S14}
\v{Z}.M. \v{S}ipu\v{s}, Translation surfaces of constant curvatures in a simply isotropic space, Period. Math. Hung. {\bf68} (2014) 160--175.

\bibitem{T1966} 
T. Tahakashi, Minimal immersions of Riemannian manifolds. J. Math. Soc. Japan {\bf18} (1966) 380--385.

\bibitem{T16}
N.C. Turgay, Lorentzian submanifolds in semi-Euclidean spaces with pointwise 1-type Gauss map, Geom. Integrability $\&$ Quantization \textbf{17} (2016) 344--359.

\bibitem{D2012} 
D.W. Yoon, Some classification of translation surfaces in Galilean 3-spaces, Int. J. Math. Anal. {\bf6} (2012) 1355--1361.

\bibitem{D2013} 
D.W. Yoon, Surfaces of revolution in the three dimensional pseudo-Galilean space, Glas. Mat. {\bf48} (2013) 415--428.

\bibitem{D2015} 
D.W. Yoon, Classification of rotational surfaces in pseudo-Galilean space, Glas. Mat. {\bf50} (2015) 453--465.

\end{thebibliography}



\end{document}